\documentclass[12pt,leqno,oneside]{amsart}
\usepackage{mathrsfs,dsfont}
\usepackage{amsmath,amstext,amsthm,amssymb,amscd,bbm}
\usepackage{charter}
\usepackage{typearea}
\usepackage{pdfsync}
\usepackage{color}
\usepackage{mathtools}
\usepackage{enumitem}

\usepackage
{hyperref}

\usepackage[width=6.4in,height=8.5in]
{geometry}

\numberwithin{equation}{section}

\newtheorem{Theorem}{Theorem}[section]
\newtheorem{Proposition}[Theorem]{Proposition}
\newtheorem{Lemma}[Theorem]{Lemma}
\newtheorem{Corollary}[Theorem]{Corollary}
\theoremstyle{definition}
\newtheorem{Definition}[Theorem]{Definition}

\newtheorem{Remark}[Theorem]{Remark}

\newcommand{\db}{\overline\partial}
\newcommand{\ov}{\overline}

\newcommand{\wi}{\widetilde}

\DeclareMathOperator{\codim}{codim}

\DeclareMathOperator{\Tr}{Tr}

\DeclareMathOperator{\rank}{rank}
\DeclareMathOperator{\Aut}{Aut}
\DeclareMathOperator{\End}{End}
\DeclareMathOperator{\Id}{Id}
\newcommand{\cali}[1]{\mathscr{#1}}
\newcommand{\cO}{\cali{O}}

\newcommand{\cT}{\cali{T}}
\newcommand{\cC}{\cali{C}}

\newcommand{\field}[1]{\mathbb{#1}}

\newcommand{\R}{\field{R}}
\newcommand{\C}{\field{C}}
\newcommand{\N}{\field{N}}

\renewcommand{\P}{\field{P}}

\newcommand{\E}{\field{E}}
\newcommand{\G}{\field{G}}

\newcommand{\FS}{\mathrm{FS}}

\newcommand{\comment}[1]{}

\begin{document}

\title{Tian's theorem for Grassmannian embeddings and degeneracy sets of random sections}

\author{Turgay Bayraktar} 
\thanks{T.\ Bayraktar is partially supported by T\"{U}B\.{I}TAK \& German DAAD Collaboration Grant ARDEB-2531/121N191}
\address{Faculty of Engineering and Natural Sciences, 
Sabanc{\i} University, \.{I}stanbul, Turkey}
\email{tbayraktar@sabanciuniv.edu}

\author{Dan Coman}
\thanks{D.\ Coman is supported by the NSF Grant DMS-2154273}
\address{Department of Mathematics, Syracuse University, 
Syracuse, NY 13244-1150, USA}
\email{dcoman@syr.edu}

\author{Bingxiao Liu}
\address{Univerisit\"at zu K\"oln, Mathematisches institut,
Weyertal 86-90, 50931 K\"oln, Germany} 
\email{bingxiao.liu@uni-koeln.de}
\thanks{B.\ Liu is supported by DFG Priority Program 2265 
`Random Geometric Systems' (Project-ID 422743078)}

\author{George Marinescu}
\address{Univerisit\"at zu K\"oln, Mathematisches institut,
Weyertal 86-90, 50931 K\"oln, Germany 
\newline\mbox{\quad}\,Institute of Mathematics `Simion Stoilow', 
Romanian Academy, Bucharest, Romania}
\email{gmarines@math.uni-koeln.de}
\thanks{G.\ Marinescu is partially supported 
by the DFG funded projects SFB TRR 191 `Symplectic Structures in Geometry, 
Algebra and Dynamics' (Project-ID 281071066\,--\,TRR 191),
DFG Priority Program 2265 `Random Geometric Systems' 
(Project-ID 422743078), the ANR-DFG project `QuasiDy\,--\,Quantization, Singularities, 
and Holomorphic Dynamics' (Project-ID 490843120)}


\date{May 25, 2025}

\begin{abstract}
Let $(X,\omega)$ be a compact K\"ahler manifold, 
$(L,h^L)$ be a positive line bundle, and $(E,h^E)$ 
be a Hermitian holomorphic vector bundle of rank $r$ on $X$. 
We prove that the pullback by the Kodaira embedding associated 
to $L^p\otimes E$ of the $k$-th Chern class of the dual of the universal 
bundle over the Grassmannian converges as $p\to\infty$ 
to the $k$-th power of the Chern form $c_1(L,h^L)$, for $0\leq k\leq r$.
If $c_1(L,h^L)=\omega$ we also determine the second term
in the semiclassical expansion, which involves $c_1(E,h^E)$.
As a consequence we show that the limit distribution of zeros of random
sequences of holomorphic sections of high powers $L^p\otimes E$
is $c_1(L,h^L)^r$. Furthermore, we compute the expectation of the currents 
of integration along degeneracy sets of random holomorphic sections. 
\end{abstract}

\maketitle

\tableofcontents


\section{Introduction}\label{S:Intro}

Let $(X,\omega)$ be a compact K\"ahler manifold of dimension $n$, $(L,h^L)$
be a positive line bundle on $X$, and $(E,h^E)$ be a Hermitian holomorphic vector
bundle of rank $r\leq n$ on $X$. For $p\geq1$, we set 
\[L^p:=L^{\otimes p},\;V_p:=H^0(X,L^p\otimes E),\;d_p:=\dim V_p-1.\]
We let $h^{L^p}$ and $h^{L^p\otimes E}$ be the Hermitian metrics induced by 
$h^L$ and $h^E$ on $L^p$, respectively on $L^p\otimes E$. We endow $V_p$
with the $L^2$-inner product induced by this metric data, namely
\begin{equation}\label{e:ip1}
(S,S')_p=\int_X\langle S,S'\rangle_{h^{L^p\otimes E}}\,\frac{\omega^n}{n!}\,,\,\;S,S'\in V_p.
\end{equation}
We endow its dual space $V_p^\star=H^{0}(X, L^{p}\otimes E)^\star$
with the inner product induced by the one on $V_p$. 

Let $\G(r,V_p^\star)$ be the Grassmannian of $r$-dimensional subspaces of $V_p^\star$. The Kodaira 
map is defined by
\begin{equation}\label{e:Kod1}
\Phi_{p}:X\to\G(d_p+1-r,V_p)=\G(r,V_p^\star)\,,\,\;\Phi_{p}(x)=\{s\in V_p:\,s(x)=0\}.
\end{equation}
Since $L$ is positive, the above map is a well-defined 
holomorphic embedding for all $p$ sufficiently large 
(see \cite[Theorem 5.1.18]{MM07}).

Let $\cT$ be the universal holomorphic vector bundle over $\G(r,V_p^\star)$ 
and $h^\cT$ be the Hermitian metric on it induced by the inner product on $V_p^\star$. 
Let $(\cT^\star,h^{\cT^\star})$ be the dual holomorphic vector bundle. 
We denote by $c_k(\cT^\star,h^{\cT^\star})$ the $k$-th Chern form
of $(\cT^\star,h^{\cT^\star})$.
Our first main result is a generalization of Tian's theorem for Chern forms of any degree.

\begin{Theorem}\label{T:Tian}
Let $(X,\omega)$ be a compact K\"ahler manifold of dimension $n$, 
$(L,h^L)$ be a positive line bundle on $X$, and $(E,h^E)$ be 
a Hermitian holomorphic vector bundle on $X$ of rank $r\leq n$. 
Then for every $0\leq k\leq r$ we have in the $\cC^\infty$ topology as $p\to\infty$,
\begin{equation}\label{e:Tian}
\frac{1}{p^k}\,\Phi_p^\star(c_k(\cT^\star,h^{\cT^\star}))=
\binom rk c_1(L,h^L)^k+O\left(\frac1p\right).
\end{equation}
Assume further that the bundle $(L,h^L)$ polarizes $(X,\omega)$,
that is $\omega=c_1(L,h)$. Then for every $0\leq k\leq r$ we have
in the $\cC^\infty$ topology as $p\to\infty$,
\begin{equation}
\frac{1}{p^k}\,\Phi_p^\star(c_k(\cT^\star,h^{\cT^\star}))=
\binom{r}{k}\omega^{k}
+\frac1p\binom{r-1}{k-1}c_1(E,h^E)\wedge \omega^{k-1}
+O\left(\frac{1}{p^2}\right).
\label{e:Tian1}
\end{equation}
\end{Theorem}
The semi-classical asymptotics \eqref{e:Tian} in the
$\cC^\infty$ topology means that 
for every $\ell\in\N$ there exists $C_\ell>0$ such that 
for $p$ sufficiently large we have
\begin{equation}\label{e:Tian-march}
\left\|\frac{1}{p^k}\,\Phi_p^\star(c_k(\cT^\star,h^{\cT^\star}))-
\binom rk c_1(L,h^L)^k\right\|_{\cC^\ell(X)}\leq \frac{C_\ell}{p}\,,
\end{equation}
and similarly for \eqref{e:Tian1}.

Tian's theorem \cite{Ti90} (the case $r=1$ and $\ell=2$, 
see also \cite{MM07,Ca99,Ru98,Z98} for the $\cC^\infty$ topology)
follows from Theorem \ref{T:Tian} 
by taking $E$ to be the trivial line bundle $X\times\C$. 
The case $k=1$ and $E$ arbitrary appears in \cite[Theorem 5.1.17]{MM07}
and is used to prove the Kodaira embedding theorem
\cite[Theorem 5.1.18]{MM07}. In recent years, Tian's theorem was generalized to 
the case of singular Hermitian holomorphic line bundles, and moreover to the setting 
of compact normal analytic spaces $X$ (see 
\cite{CM15,CM13,CMM17,CMN16,CMN18,CMN24}).
 
Theorem \ref{T:Tian} establishes that for an arbitrary positive line bundle $(L,h^L)$, 
we can determine the first-order approximation of the pullback of the 
$k$-th Chern form in the semi-classical limit. 
This approximation is up to a factor the $k$-th power of the Chern curvature $c_1(L,h^L)$.
Furthermore, if $c_1(L,h^L)=\omega$, we compute the second-order 
approximation of the pullback of the $k$-th Chern form.
In the semi-classical limit, we recover the first Chern form 
$c_1(E,h^E)$ of $(E,h^E)$.

An important application of Tian's theorem \cite{Ti90} is to the study of the asymptotic distribution 
of zeros of random sequences of sections of powers of a line bundle. This started with the work of 
Shiffman and Zelditch \cite{ShZ99}. See \cite{BCM,Sh08,B6,DMS12,DMM16,DMN17,DLM25} for interesting 
results in this direction. One motivation for this study comes from the particular case of random 
polynomials and the distributions of their zeros, which has a long history. We refer the reader 
to the papers \cite{B9,BL13,BCHM} and the references therein. 

As in the case of line bundles, Theorem \ref{T:Tian} 
can be applied to prove the equidistribution of zeros of 
random sections in $V_p$. The setting is as follows. 
For the generic section $s\in V_p$ its zero set $Z_s=\{s=0\}$ 
is a complex submanifold of $X$ of codimension $r$ 
(see Section \ref{S:MTG}). We denote by $[s=0]$ 
the current of integration over $Z_s$. We endow the 
projective space $\P V_p$ with the probability measure 
$\Upsilon_p=\omega_\FS^{d_p}\,$, where $\omega_\FS$ 
denotes the Fubini-Study form on $\P V_p$. 
We then consider the product probability space 
\begin{equation}\label{e:prob}
(\mathcal H,\Upsilon)=
\Big(\prod_{p=1}^\infty\P V_p\,,\prod_{p=1}^\infty\Upsilon_p\Big).
\end{equation}
Our second main result describes the distribution of zeros of 
random sequences of sections of $L^p\otimes E$ and 
gives an estimate on the speed of convergence.
An important ingredient of the proof is to consider the Kodaira map
as a meromorphic transform in the sense of Dinh-Sibony \cite{DS06}.
\begin{Theorem}\label{T:zeros}
Let $(X,\omega)$ be a compact K\"ahler manifold of dimension $n$, 
$(L,h^L)$ be a positive line bundle on $X$, and $(E,h^E)$ 
be a Hermitian holomorphic vector bundle of rank $r\leq n$ on $X$. 
Then there exist $C>0$ and $p_0\in\N$ such that the following holds: 

For any $\gamma>1$ there exist subsets 
$\mathcal{E}_p=\mathcal{E}_p(\gamma)\subset\P V_p$ such that for $p>p_0$ we have 

(i) $\Upsilon_p(\mathcal{E}_p)\leq Cp^{-\gamma}$;

(ii) if $s_p\in\P V_p\setminus \mathcal{E}_p$ then
\begin{equation}\label{e:espeed}
\Big|\Big\langle\frac{1}{p^r}\,[s_p=0]-c_1(L,h^L)^r,\varphi\Big\rangle\Big|
\leq C\gamma\,\frac{\log p}{p}\,\|\varphi\|_{\cC^2(X)},
\end{equation}
for any $(n-r,n-r)$ form $\varphi$ of class $\cC^2$ on $X$.
Moreover, the estimate \eqref{e:espeed} holds for $\Upsilon$-a.e.\ sequence 
$\{s_p\}_{p\geq1}\in\mathcal H$ provided
that $p$ is large enough. Hence $p^{-r}\,[s_p=0]\to c_1(L,h^L)^r$ in the weak 
sense of currents as $p\to\infty$, almost surely.
\end{Theorem}
By taking the Hermitian vector bundle $(E,h^E)$
to be the trivial vector bundle $X\times\C^r$ endowed
with the trivial metric we recover the results by 
Shiffman-Zelditch \cite{ShZ99} and Dinh-Sibony \cite{DS06} that
the normalized simultaneous zero currents 
$p^{-r}\,[s_{p,1}=\ldots=s_{p,r}=0]$
of $r$ independent random sections of $H^0(X,L^p)$ almost surely
converge weakly to $c_1(L,h^L)^r$, for
$1\leq r\leq n$. We refer to Section \ref{S:tuples} for
more details. 

We turn now our attention to the expectation of the currents 
of integration along degeneracy sets of random holomorphic sections. 
This topic is closely related to the theorems stated above. 
We consider the following general setting:  

\medskip

(A) $(X,\omega)$ is a compact K\"ahler manifold of dimension $n$, $(E,h^E)$
is a Hermitian holomorphic vector bundle of rank $r$ on $X$, and
\[1\leq r\leq n,\;V:=H^0(X,E),\;N:=\dim V-1,\;\P V=\P H^0(X,E).\]

\smallskip

(B) For every section $S\in V$, $Z_S:=\{x\in X:\,S(x)=0\}\neq\varnothing$.

\smallskip

(C) $N\geq r$ and for every $x\in X$, $V$ spans the fiber $E_x$ of $E$ over $x$.

\medskip

Note that the set $Z_s=Z_S$ is well-defined for $s\in\P V$, 
as it is independent of the chosen representative $S\in V\setminus\{0\}$ of $s$. 
We endow $V$ with the $L^2$-inner product determined by $h^E$ and the volume
form $\omega^n/n!$ on $X$, 
\begin{equation}\label{e:ip2}
(S,S')=\int_X\langle S,S'\rangle_{h^E}\,\frac{\omega^n}{n!}\,,\,\;S,S'\in V.
\end{equation}
We consider the dual $V^\star$ with the inner product determined
by the one on $V$, and we let $\omega_\FS$ be the induced Fubini-Study
form on the $N$-dimensional projective space $\P V$. 

Let $\G(k,V)$ denote the Grassmannian of (complex) $k$-planes in $V$. 
By (C), the Kodaira map 
\begin{equation}\label{e:Kod2}
\Phi_E:X\to\G(N+1-r,V)=\G(r,V^\star),\;\Phi_E(x)=\{s\in V:\,s(x)=0\},
\end{equation}
is well-defined and holomorphic. Let $\cT\to \mathbb{G}(r,V^\star)$ 
be the universal bundle endowed with the Hermitian metric $h^{\cT}$ 
induced by the inner product on $V^\star$, and let 
$(\cT^\star,h^{\cT^\star})$ be the dual holomorphic vector bundle.

Let $1\leq k\leq r$. We will denote by $\Lambda^kE$ the 
$k$-th exterior power of $E$. In particular $\Lambda^rE$
is the determinant bundle of $E$. If $S_1,\ldots,S_k\in V$ then 
$S_1\wedge\ldots\wedge S_k\in H^0(X,\Lambda^kE)$. 
Given $(s_1,\ldots,s_k)\in(\P V)^k$ (or $(s_1,\ldots,s_k)\in V^k$), 
the degeneracy set 
\begin{equation}\label{e:degset}
D_k(s_1,\ldots,s_k)=\{x\in X:\,(s_1\wedge\ldots\wedge s_k)(x)=0\}
\end{equation}
is the set of points $x$ where $s_1(x),\ldots,s_k(x)$ are linearly dependent 
(see \cite[p.\ 411]{GH94}). 

Let $\mu_k$ be a probability measure on $(\P V)^k$ and  
$\E_k(\mu_k)=\E([D_k(s_1,\ldots,s_k)],\mu_k)$ 
be the expectation current of the current valued random variable 
\[(\P V)^k\ni(s_1,\ldots,s_k)\longmapsto[D_k(s_1,\ldots,s_k)],\] 
defined by 
\begin{equation}\label{e:expk}
\langle \E_k(\mu_k),\phi \rangle=
\int_{(\P V)^k}\langle[D_k(s_1,\ldots,s_k)],\phi\rangle\,d\mu_k\,,
\end{equation}
where $\phi$ is a smooth $(n+k-r-1,n+k-r-1)$ form on $X$. 
By using the fact that the Poincar\'e dual
of $D_k(s_1,\ldots,s_k)$ is $c_{r+1-k}(E)$ 
we find the cohomology class of the expectation 
current for general probability measures.
\begin{Theorem}\label{T:expcoh}
Let $(X,\omega),\,(E,h^E)$ verify (A)-(C).
For $1\leq k\leq r$, let $\mu_k$ be a probability measure 
on $(\P V)^k$ such that $\mu_k(A)=0$ for any proper analytic subset 
$A\subset(\P V)^k$. Then $\E_k(\mu_k)$ is a well-defined 
positive closed current of bidegree $(r+1-k,r+1-k)$ on $X$ 
and it belongs to the Chern class $c_{r+1-k}(E)$.
\end{Theorem}
In the case of Gaussians or Fubini-Study volumes we obtain 
a precise formula for the expectation currents.

\begin{Theorem}\label{T:expdeg}
Let $(X,\omega),\,(E,h^E)$ verify (A)-(C). 
Assume that the Kodaira map $\Phi_E:X\to\G(r,V^\star)$ 
defined in \eqref{e:Kod2} is an embedding. 
For $1\leq k\leq r$, let $\nu_k$ be the Gaussian probability measure 
on $V^k$ and $\mu_k$ be the product measure on $(\P V)^k$
determined by the Fubini-Study volume $\omega_\FS^N$ on $\P V$.
Then the expectation currents $\E_k(\mu_k), \E_k(\nu_k)$ are given by 
\[\E_k(\mu_k)=\E_k(\nu_k)=
\Phi_E^\star\big(c_{r+1-k}(\cT^\star,h^{\cT^\star})\big),\;1\leq k\leq r.\]
\end{Theorem}

Combining Theorems \ref{T:Tian} and \ref{T:expdeg}, we obtain 
asymptotic equidistribution results for the degeneracy sets. 
\begin{Corollary}\label{cor:1.5}
Let $(X,\omega)$ be a compact K\"ahler manifold of dimension $n$, 
$(L,h^L)$ be a positive line bundle on $X$, and $(E,h^E)$ 
be a Hermitian holomorphic vector bundle of rank $r\leq n$ on $X$. For $p\in \mathbb{N}$ and $1\leq k\leq r$, let $\nu^p_k$ be the Gaussian probability measure 
on $V_p^k$ and $\mu^p_k$ be the product measure on $(\P V_p)^k$
determined by the Fubini-Study volume $\Upsilon_p$ on $\P V_p$.
Then we have as $p\to\infty$,
\begin{equation}
\label{eq:cor1.5}
    \frac{1}{p^{r+1-k}}\,\E_k(\mu^p_k)=\frac{1}{p^{r+1-k}}\,\E_k(\nu^p_k)=\binom{r}{k-1} c_1(L,h^L)^{r+1-k}+O\left(\frac1p\right).
\end{equation}
If we assume in addition that $\omega=c_1(L,h^L)$, then
\begin{equation}
\label{eq:cor1.5-2}
\begin{split}
        \frac{1}{p^{r+1-k}}\,\E_k(\mu^p_k)&=\frac{1}{p^{r+1-k}}\,\E_k(\nu^p_k)\\
        &=\binom{r}{k-1} \omega^{r+1-k}+\frac{1}{p}\binom{r-1}{r-k}c_1(E,h^E)\wedge \omega^{r-k}+O\left(\frac{1}{p^2}\right).
\end{split}
\end{equation}
\end{Corollary}
This paper is organized as follows. 
After some preliminary material in Section \ref{S:Prelim}, 
we present the proof of the approximation theorem for the pullback 
of Chern forms by the Kodaira map in Section \ref{S:Tian}.
In Section \ref{S:MTG}, we show that the Kodaira map can 
be interpreted as a meromorphic transform in the sense of Dinh-Sibony.
In Section \ref{S:expectation}, we compute the expectation of currents 
of integration along degeneracy sets of random holomorphic sections.
In Section \ref{S:zeros}, we establish that the limit distribution of zeros 
of random sequences of holomorphic sections of high powers 
$L^p\otimes E$ is $c_1(L,h^L)^r$. Finally, in Section \ref{S:tuples}, 
we specialize our results to the case of simultaneous zeros of several random holomorphic 
sections of $L^p$.

\section{Preliminaries}\label{S:Prelim}
In this section we gather the necessary material regarding Chern forms, universal bundles 
on Grassmannians, and Bergman kernels.
\subsection{Chern forms of Hermitian vector bundles}
Let $A$ be an $r\times r$ matrix of complex numbers. Recall that the {\em elementary invariant polynomials} $P^k$ are defined by (see, e.g., \cite[p.\ 402]{GH94}) 
\[\det(A+t\cdot\mathrm{Id}_r)=\sum_{k=0}^rP^{r-k}(A)t^k,\;t\in\C,\]
where $\mathrm{Id}_r$ is the $r\times r$ identity matrix. 
We have that $P^k(U^{-1}AU)=P^k(A)$ for any invertible $r\times r$ matrix $U$ and 
\begin{equation}\label{e:symmpol}
P^k(A)=\sum_{\sharp I=k}\det A_{I,I},
\end{equation}
where $A_{I,J}$ denotes the $(I,J)$-th minor $(A_{ij})_{i\in I,j\in J}$ of $A$. 

Since the wedge product is commutative on differential forms of degree $2$, $P^k(A)$
is well-defined in the same way for any $r\times r$ matrix $A$ of forms of degree $2$, namely
\[\det(A+t\cdot\mathrm{Id}_r)=\sum_{k=0}^rP^{r-k}(A)\wedge t^k.\]
where $t$ is a form of degree $2$ or a number.
Then 
\begin{equation}\label{e:Chern-m}
c_k(A):=P^k(A)
\end{equation} 
is a form of degree $2k$ given by \eqref{e:symmpol}.

\medskip

Let $(X,\omega)$ be a compact K\"ahler manifold of dimension $n$ and 
$(E,h^{E})$ be a holomorphic vector bundle on $X$ of rank $r\geq 
1$. Let $\nabla^{E}$ denote the associated Chern connection on $E$, 
and let $R^{E}=(\nabla^{E})^{2}$ be its curvature.

\begin{Definition}
The (total) Chern form associated to a Hermitian holomorphic vector 
bundle $(E,h^{E})$ of rank $r$ is defined by
\begin{equation}\label{e:Chern1}
c(E,h^{E})=\det\Big(\frac{i}{2\pi}R^{E}+\mathrm{Id}_{E}\Big)\in
\bigoplus_{k=0}^{r} \Omega^{(k,k)}(X,\R).
\end{equation}
\end{Definition}

The differential form $c(E,h^{E})$ is closed, thus it defines a de 
Rham cohomology class $c(E)\in H^{\bullet}_{\mathrm{dR}}(X,\R)$, which 
is independent of the choice of $h^{E}$.
The $k$-th Chern form $c_k(E,h^{E})$ is the component of bidegree $(k,k)$ of
$c(E,h^{E})$. Then
\begin{equation}\label{e:Chern2}
c(E,h^{E})=\sum_{k=0}^{r}c_k(E,h^{E})\,,\,\;c_k(E,h^E)
=c_k\Big(\frac{i}{2\pi}\,R^E\Big).
\end{equation}
In particular, 
\begin{equation}
\begin{split}
&c_{0}(E,h^{E})=1,\\
&c_{1}(E,h^{E})=\frac{i}{2\pi}\Tr^{E}[R^{E}],\\
&c_{2}(E,h^{E})=\frac{1}{8\pi^{2}}(\Tr^{E}[R^{E,2}]-
\Tr^{E}[R^{E}]^{2}),\\
&c_{r}(E,h^{E})=\left(\frac{i}{2\pi}\right)^{r}\det R^{E}.
\end{split}
\label{eq:3.2.3}
\end{equation}
If $(L,h^{L})$ is a Hermitian holomorphic line bundle on 
$X$ then $\rank(L\otimes E)=r$ and 
\begin{equation}
c(L\otimes E,h^{L\otimes E})=
\sum_{k=0}^{r}c_k(E,h^{E})\wedge c(L,h^{L})^{r-k},
	\label{eq:3.2.4}
\end{equation}
where, by definition, $c(L,h^{L})=1+c_{1}(L,h^{L})$. Then for 
$k=0,1,\ldots,r$, we have
\begin{equation}
c_k(L\otimes E,h^{L\otimes E})=
 \sum_{j=0}^k\binom{r-j}{k-j}c_j(E,h^{E})\wedge c_{1}(L,h^{L})^{k-j}.
	\label{eq:3.2.4bis}
\end{equation}
For every integer $p\geq 1$, let $(L^p, h^{L^p}):=
(L^{\otimes p}, (h^L)^{\otimes p})$.

\begin{Lemma}\label{L:ChernLp}
	For every $\ell\in\N$ there exists
	$C_\ell>0$ such that if $p\geq 1$ and 
	$0\leq k\leq r$ then 
	\begin{equation}
\left\|\frac{1}{p^k}c_k(L^{p}\otimes E, h^{L^{p}\otimes E})-
\binom rk c_1(L,h^L)^k\right\|_{\mathscr{C}^\ell(X)}\leq 
\frac{C_\ell}{p}\,.
	\end{equation}
\end{Lemma}

\begin{proof}
Since $c_{1}(L^{p},h^{L^p})=pc_{1}(L,h^{L})$, 
we get by \eqref{eq:3.2.4bis}
\begin{align*}
c_k(L^{p}\otimes E,h^{L^{p}\otimes E})&=
\sum_{j=0}^k\binom{r-j}{k-j} p^{k-j} c_j(E,h^{E})\wedge c_1(L,h^L)^{k-j}\\
&=p^k\binom rk c_1(L,h^L)^k+O(p^{k-1}).
\end{align*}
The lemma now follows directly from the above identity.
\end{proof}

\subsection{Universal bundle on Grassmannian}\label{SS:universal}
The Grassmannian $\mathbb{G}(r,m)$ is the set of 
$r$-dimensional complex linear subspaces of $\C^{m}$ ($r\leq m$). The universal 
(or tautologigal) vector bundle $\cT$ on $\mathbb{G}(r,m)$ 
is the holomorphic vector bundle of rank $r$ defined by 
\begin{equation}
	\cT=\left\{(W,f)\in \mathbb{G}(r,m)\times \C^{m}\;:\; 
	f\in W\subset\C^{m}\right\}.
\end{equation}
Let $h^{\cT}$ denote the Hermitian metric on $\cT$ 
induced by the standard Hermitian metric $h^{\C^{m}}$ on $\C^{m}$.
Let $\nabla^{\cT}$, $R^{\cT}$ denote 
the Chern connection, respectively the Chern curvature, of $(\cT, 
h^{\cT})$. We recall here the explicit formula for $R^{\cT}$ in certain local charts.

Let $W_{0}\in \mathbb{G}(r,m)$. We fix $\{e_{j}\}_{j=1}^{r}$ an 
orthonormal $\C$-basis of $(W_{0}, h^{\C^{m}}|_{W_{0}})$, and we extend it to an orthonormal basis 
$\{e_{j}\}_{j=1}^{m}$ of $\C^{m}$. 
The map 
$$(z_{j\ell})_{1\leq j\leq r,\; 1\leq \ell\leq m-r}\in 
\C^{r\times(m-r)}\mapsto \mathrm{Span}_{\C}\{e_{j}+
\sum_{\ell=1}^{m-r} z_{j\ell}e_{\ell+r}, 
j=1,\ldots,r\}\subset \C^{m},$$ 
is a local chart of $\mathbb{G}(r,m)$ 
near the point $W_{0}$. Set $Z=(z_{j\ell})_{{1\leq j\leq r,\; 1\leq \ell\leq m-r}}$, for 
the coordinate matrix of size $r\times (m-r)$.
For $(W,f)\in \cT$ and $W$ near $W_{0}$ set
\begin{equation}
	\eta_{j}(W)=e_{j}+\sum_{\ell=1}^{m-r} z_{j\ell}e_{\ell+r}\in\C^{m}\,,\,\;j=1,\ldots, r.
	\label{eq:2.11SG}
\end{equation}
Then $\{\eta_{j}\}_{j=1}^{r}$ is a local 
holomorphic frame of $\cT$ near $W_{0}$.

Let $K(W)=(K_{kj}(W))_{1\leq k,j \leq r}$ be the square matrix valued 
function near $W_{0}$ given by
\begin{equation}
	K_{kj}(W)=h^{\cT}(\eta_{j},\eta_{k}).
\end{equation}
Then
\begin{equation}
	K(W)=I+ZZ^\star,
\end{equation}
where $I=\mathrm{Id}_r$, $Z^\star=\bar{Z}^{\mathrm{T}}$, and  
$A^{\mathrm{T}}$ denotes the transpose of a matrix $A$.

Under this local holomorphic frame, the Chern connection of 
$(\cT, h^{\cT})$ is given by
\begin{equation}
	\nabla^{\cT}=d+K^{-1}\partial K.
\end{equation}
Then the Chern curvature is
\begin{equation}
	R^{\cT}=K^{-1}\bar{\partial}\partial K - 
	K^{-1}\bar{\partial}K \wedge K^{-1}\partial K.
\end{equation}
With the notations $dZ=(dz_{j\ell}), dZ^\star=(d\bar{z}_{j\ell})^{\mathrm{T}}$ 
we have
\begin{equation}
	R^{\cT}=-(I+ZZ^\star)^{-1}dZ\wedge 
	dZ^\star-(I+ZZ^\star)^{-1}ZdZ^\star\wedge(I+ZZ^\star)^{-1}dZ Z^\star.
\end{equation}
In particular, at $W_{0}$,
\begin{equation}
	R^{\cT}(W_{0})_{kj}=-\sum_{\ell=1}^{m-r} dz_{k\ell}\wedge 
	d\bar{z}_{j\ell}.
\end{equation}
Let $\cT^{\ast}$ be the dual holomorphic vector bundle of 
$\cT$ on $\mathbb{G}(r,m)$, and let $h^{\cT^{\ast}}$ 
be the induced Hermitian metric. For $j=1,\ldots, r$, 
$(W,f)\in\cT$, set
\begin{equation}
\left(\sigma_{j}(W),f\right)=(e^{j},f).
\label{eq:2.18SG}
\end{equation}
Then $\{\sigma_{j}\}_{j=1}^{r}$ is a local holomorphic frame of 
$\cT^{\ast}$ near $W_{0}$. Note that
\begin{equation}
	(\sigma_{i},\eta_{j})=\delta_{ij},
\end{equation}
i.e., $\{\sigma_{j}\}_{j=1}^{r}$ is exactly the dual frame of 
$\{\eta_{j}\}_{j=1}^{r}$.

Let $H(W)=(H_{kj}(W))_{1\leq k,j \leq r}$ be the square matrix valued 
function near $W_{0}$ given by
\begin{equation}
	H_{kj}(W)=h^{\cT^{\ast}}(\sigma_{j},\sigma_{k}).
\end{equation}
Then
\begin{equation}
H(W)=(K(W)^{-1})^{\mathrm{T}}=(I+\bar{Z}Z^{\mathrm{T}})^{-1}.
\end{equation}
Under this local holomorphic frame of $\cT^{\ast}$, the Chern connection of 
$(\cT^{\ast}, h^{\cT^{\ast}})$ is given by
\begin{equation}
	\nabla^{\cT^{\ast}}=d+H^{-1}\partial H.
	\label{eq:2.22New}
\end{equation}
Then the Chern curvature is
\begin{equation}\label{eq:2.22New1}
	R^{\cT^{\ast}}=H^{-1}\bar{\partial}\partial H - 
	H^{-1}\bar{\partial}H \wedge H^{-1}\partial H.
\end{equation}
In terms of the local coordinate $Z$ (viewed as a matrix of size 
$r\times(m-r)$),
\begin{equation}
	R^{\cT^{\ast}}=-d\bar{Z}\wedge 
	(dZ)^{\mathrm{T}}(I+\bar{Z}Z^{\mathrm{T}})^{-1}-\bar{Z}(dZ)^{T}(I+\bar{Z}Z^{\mathrm{T}})^{-1}\wedge d\bar{Z}Z^{\mathrm{T}}(1+\bar{Z}Z^{\mathrm{T}})^{-1}.
\end{equation}
Note that  
\begin{equation}
	R^{\cT^{\ast}}=-(R^{\cT})^{\mathrm{T}}.
\end{equation}
In particular, at $W_{0}$,
\begin{equation}
	R^{\cT^{\ast}}(W_{0})_{kj}=\sum_{\ell=1}^{m-r} dz_{j\ell}\wedge 
	d\bar{z}_{k\ell}.
\end{equation}
Now we can write the Chern forms for $\cT^\ast$ at the point $W_0$ in the above coordinates. 
By \eqref{e:Chern2} and \eqref{e:symmpol} we have for $k=1,\ldots, r$,
\begin{equation}
c_k(\cT^\ast,h^{\cT^\ast})(W_0)=\frac{i^k}{(2\pi)^k}\sum_{\sharp I=k}\det R^{\cT^{\ast}}(W_{0})_{I,I},
\label{eq:2.27-chern}
\end{equation}
where $R^{\cT^{\ast}}(W_{0})_{I,J}$ is the  $(I,J)$-th minor $(R^{\cT^{\ast}}(W_{0})_{ij})_{i\in I,j\in J}$ of the matrix $R^{\cT^{\ast}}(W_{0})$.


\subsection{Bergman kernel asymptotics}\label{SS:Bergman}

Let $(X,\omega)$ be a compact K\"{a}hler manifold of dimension $n$, 
and let $(L,h^{L})$ be a positive holomorphic line bundle on $X$, 
that is, $h^{L}$ is a smooth Hermitian metric on $L$ 
such that the first Chern form $c_{1}(L,h^{L})$ is a K\"ahler form 
on $X$. Moreover, $c_{1}(L,h^{L})$ is a de Rham
representative of the Chern class $c_{1}(L)\in H^{2}(X, \R)$.
We identify the 2-form $R^L$ with the Hermitian matrix
$\dot{R}^L \in \End(T^{(1,0)}X)$ such that
for $W,Y\in T^{(1,0)}X$,
$R^L (W,\ov{Y}) = \langle \dot{R}^LW, \ov{Y}\rangle$.

Here the K\"{a}hler form $\omega$ is not necessarily 
equal to $c_{1}(L,h^{L})$. We will call $(L,h^{L})$ 
a prequantum line bundle of $(X,\omega)$ if the following condition holds
\begin{equation}
	\omega=c_{1}(L,h^{L})=\frac{i}{2\pi} R^{L}.
	\label{eq:3.1.3}
\end{equation}

Let $(E,h^{E})$ be a holomorphic vector bundle on $X$ of rank $r\geq 
1$. Let $\nabla^{E}$ denote the associated Chern connection on $E$, 
and let $R^{E}=(\nabla^{E})^{2}$ be its curvature.

Let $\mathcal{L}^2(X, L^{p}\otimes E)$ be the $L^2$-space
corresponding to the $L^2$ inner product on $X$ induced by $g^{TX}$, 
$h^{L^{p}}$ and $h^{E}$.
Let $H^{0}(X, L^{p}\otimes E)$ denote the space of holomorphic 
sections of $L^{p}\otimes E$ over $X$, which is a finite dimensional 
complex vector subspace of $\mathcal{L}^2(X, L^{p}\otimes E)$.
Let 
\begin{equation}
P_p:\mathcal{L}^2(X, L^{p}\otimes E)\to
H^{0}(X, L^{p}\otimes E)
	\label{eq:3.1.4}
\end{equation}
be the orthogonal (Bergman) projection.
The Schwartz kernel $P_p(\cdot,\cdot)$ of $P_p$
with respect to the volume form $\frac{\omega^n}{n!}$ 
is called the Bergman kernel of $H^{0}(X, L^{p}\otimes E)$.
It is a smooth section of $(L^p\otimes E)\boxtimes(L^p\otimes E)$
over $X\times X$, 
\begin{equation}\label{bk2.2}
P_p(x,x') \in (L^p\otimes E)_x\otimes (L^p\otimes E)_{x'}^*\,,
\end{equation}
especially, 
\begin{equation}\label{bk2.3}
P_p(x,x) \in \End(L^p\otimes E)_x = \End(E)_x,
\end{equation}
where we use the canonical identification $\End(L^p)=\C$ 
for any line bundle $L$ on $X$.
Let $\{S^p_j\}_{j=0}^{d_p}$, $d_p := \dim H^{0}(X,L^p\otimes E)-1$, 
be any orthonormal basis of
$H^{0}(X,L^p\otimes E)$ with respect to the $L^2$ inner
product.
Then
\begin{equation} \label{bk2.4}
P_p(x,x')=\sum_{j=0}^{d_p} S^p_j (x) \otimes S^p_j(x')^*
\in (L^p\otimes E)_x\otimes (L^p\otimes E)_{x'}^*.
\end{equation}
and 
\begin{equation} \label{bk2.5}
P_p(x,x)= \sum_{j=0}^{d_p} S^p_j (x) \otimes S^p_j(x)^*
\in \End(E_x).
\end{equation}
We have the following diagonal asymptotic expansion
of the Bergman metric. 
\begin{Theorem}[{\cite[Theorem 4.1.1]{MM07}}]\label{bkt2.1} 
For every $m\in\N$ there exist smooth coefficients
$\boldsymbol{b}_m(x)\in\End (E)_x$,
which are polynomials in $R^{TX}$, 
$R^E$ {\rm (}and $d\Theta$, $R^L${\rm )} and their derivatives of order
$\leqslant 2m-2$  {\rm (}resp. $2m-1$, $2m${\rm )} and reciprocals
of linear combinations of eigenvalues of $\dot{R}^L$ at $x$, 
such that
\begin{align}\label{abk2.5}
\boldsymbol{b}_0={\det}(\dot{R}^L/(2\pi)) \Id_{E},
\end{align}
and for any $k,\ell\in\N$, there exists
$C_{k,\ell}>0$ such that for any $p\in \N$,
\begin{align}\label{bk2.6}
&\Big \|P_p(x,x)- \sum_{m=0}^{k}\boldsymbol{b}_m(x)
p^{n-m} \Big \|_{\cC^\ell(X)} 
\leqslant C_{k,\ell}\, p^{n-k-1}.
\end{align}
\end{Theorem}
In \cite[Theorem 4.1.1]{MM07} additional information about
the uniformity of the expansion with respect to 
$g^{TX}$, $h^L$, and $h^E$ and their derivatives is given.
We refer to \cite{DLM06,MM15} for further interesting results on the asymptotic 
expansion of the Bergman kernels for powers of a holomorphic line bundle, or 
more generally for an arbitrary sequence of holomorphic line bundles \cite{CLMM}.

\subsection{Dimension of the space of holomorphic sections}
Let $(X,\omega)$ be a compact K\"{a}hler manifold of
dimension $n$, 
and let $(L,h^{L})$ be a positive line bundle on $X$. 
Let $(E,h^{E})$ be a holomorphic vector bundle on $X$ of rank $r\geq1$.
For $p$ large enough the Kodaira-Serre vanishing theorem 
\cite[Theorem 1.5.6]{MM07}
says that $H^q(X,L^p\otimes E)=0$ for every $q\geq1$ 
and thus, the Euler number 
$\chi(X,L^p\otimes E)=\sum_{q=0}^n(-1)^q\dim H^q(X,L^p\otimes E)$ equals
$\dim H^0(X,L^p\otimes E)$.
On the other hand, by the Riemann-Roch-Hirzebruch theorem \cite[Theorem 4.9]{BGV}
we have
\[\chi(X,L^p\otimes E)=
\int_X \operatorname{td}(T X) \operatorname{ch}(L^p\otimes E),\]
where $\operatorname{td}(T^{1,0}X)$ denotes the Todd class
of $T^{1,0}X$ and $\operatorname{ch}(F)$ denotes the Chern character
of a Hermitian vector bundle $F$ (see e.g.\ \cite[Appendix B.5]{MM07}).
Hence, we have for $p$ large enough,
\begin{equation}\label{abk2.3}
\begin{split}
\dim H^{0}(X,L^p\otimes E)&=\chi(X,L^p\otimes E)=
\int_X \operatorname{td}(T X) \operatorname{ch}(L^p\otimes E)\\
&= r\int_X \frac{c_1(L)^n}{n!} \;p^n +R_{n-1}(p),
\end{split}
\end{equation}
where $R_{n-1}(p)$ is a polynomial of degree $(n-1)$ in $p$.
Consequently, $\dim H^{0}(X,L^p\otimes E)$ is a polynomial of degree $n$ in $p$
with positive leading coefficient $r\int_X c_1(L)^n/n!$.

Furthermore, one can also determine the order of growth of $\dim H^{0}(X,L^p\otimes E)$
by utilizing the identity
\[\dim H^{0}(X,L^p\otimes E)=
\int_X \operatorname{Tr}_E[ P_p(x,x)] dv_X(x),\]
and integrating the expansion \eqref{bk2.6} to obtain, for sufficiently large $p$,
\[\dim H^{0}(X,L^p\otimes E)=r\int_X \frac{c_1(L)^n}{n!} \;p^n+O(p^{n-1}).\]

\section{Tian's theorem for Grassmannian embeddings}\label{S:Tian}

We give in this section the proof of Theorem \ref{T:Tian}. 
Recall from \eqref{e:Kod1} the definition of the 
Kodaira map $\Phi_{p}:X\to\G(r,V_p^\star)$, where 
$V_p=H^{0}(X, L^{p}\otimes E)$. 
Since $L$ is positive this map is a holomorphic  
embedding for all $p$ sufficiently large (see \cite[Theorems 5.1.15--5.1.18]{MM07}).
Recall that  $\cT$ is the universal holomorphic vector bundle over 
$\G(r,V_p^\star)$ endowed with the Hermitian metric $h^{\cT}$  induced by the
inner product on $V_p^\star$. Let $(\cT^\star,h^{\cT^\star})$ be the corresponding 
dual holomorphic vector bundle. 

If $s\in V_p$ the evaluation map $\varepsilon_s:V_p^\star\to\C$, 
$\varepsilon_s(f)=f(s)$, defines a holomorphic 
section $\sigma_{s}\in H^0(\G(r,V_p^\star),\cT^\star)\equiv V_p\,$, 
where $\sigma_s(W)=\varepsilon_s|_W$, 
$W\in\G(r,V_p^\star)$.
For $p$ large enough we have the following isomorphism of 
holomorphic vector bundles on $X$ (see  \cite[Theorems 5.1.16]{MM07}),
\begin{equation}
	\begin{split}
&\Psi_{p}:\Phi_{p}^\star(\cT^\star)\rightarrow \;L^{p}\otimes E,\\
&\Psi_{p}((\Phi_{p}^\star\sigma_{s})(x))=s(x),\;\text{ for any $s\in V_p$, $x\in X$.}
	\end{split}
	\label{e:isostar}
\end{equation}
Moreover, under this isomorphism, we have
\begin{equation}
	h^{\Phi_{p}^\star(\cT^\star)}(x)=h^{L^{p}\otimes E}(x)\circ 
	P_{p}(x,x)^{-1},
	\label{e:isostarP}
\end{equation}
where $h^{\Phi_{p}^\star(\cT^\star)}$ is the Hermitian metric on 
$\Phi_{p}^\star(\cT^\star)$ induced by $h^{\cT^\star}$.

Assertion \eqref{e:Tian} of Theorem \ref{T:Tian} is a direct consequence of Theorem \ref{T:Tian2} 
below and Lemma \ref{L:ChernLp}.

\begin{Theorem}\label{T:Tian2}
Let $(X,\omega)$ be a compact K\"ahler manifold of dimension $n$, $(L,h^L)$ 
be a positive line bundle on $X$, and $(E,h^E)$ be a 
Hermitian holomorphic vector bundle on $X$ of rank $r\leq n$. 
Then for every $\ell\in\N$ there exists $C_\ell>0$ such that 
\begin{equation}\label{e:Tian2}
\big\|\Phi_p^\star(c_k(\cT^\star,h^{\cT^\star}))-
c_{k}(L^{p}\otimes E, h^{L^{p}\otimes E})
\big\|_{\cC^\ell(X)}\leq C_\ell\,p^{k-1}
\end{equation}
holds for $0\leq k\leq r$, and for all $p$ sufficiently large.
\end{Theorem}
\begin{proof}
We assume throughout the proof that $p$ is sufficiently large. By \cite[(5.1.38)]{MM07}, 
the evaluation map $s\in V_p\mapsto s(x)\in(L^{p}\otimes E)_{x}$ is surjective at any point $x\in X$.

We fix a point $x_0\in X$ and prove that \eqref{e:Tian2} holds in a neighborhood of $x_0$. 
Let $e_{L}$ be a holomorphic frame of $L$ near $x_{0}$ and $v_{1},\ldots, v_{r}$ be a 
holomorphic frame  of $E$ near $x_{0}$. Set 
\[t_{j}=v_{j}\otimes e_{L}^{\otimes p}\,,\,\;j=1,\ldots,r.\]
Then $t_{1},\ldots, t_{r}$ is a holomorphic frame  of $L^p\otimes E$ near $x_{0}$. 
Using the isomorphism $\Psi_{p}$ given in \eqref{e:isostar} we see that there is a holomorphic 
frame $\tau_{1},\ldots, \tau_{r}$ of $\Phi_{p}^\star(\cT^\star)$ near $x_0$ such that 
\[\Psi_{p}(\tau_j(x))=t_j(x),\;j=1,\ldots,r,\;x\text{ near }x_0.\]
Using \eqref{e:isostarP} we infer that 
\begin{equation}\label{e:isostarP1}
h^{\Phi_{p}^\star(\cT^\star)}(\tau_{m}(x),\tau_{j}(x))=
h^{L^{p}\otimes E}\big(P_p(x,x)^{-1}t_{m}(x),t_{j}(x)\big).
\end{equation}

In order to prove \eqref{e:Tian2}, we will relate $R^{\Phi_{p}^\star(\cT^\star)}$ 
to $R^{L^p\otimes E}$ via 
\eqref{e:isostarP1}. For sufficiently large $p$, 
$P_p(x,x)\in\mathrm{End}(E_x)$ is invertible. 
For $x$ near $x_{0}$, we define the matrix valued function $Q_p(x)$ by
	\begin{equation}
		P_{p}(x,x)^{-1}t_{j}(x)=\sum_{k=1}^r(Q_{p}(x))_{kj} t_{k}(x).
		\label{eq:3.2.16}
	\end{equation}
This is the matrix representing the inverse of $P_{p}(x,x)$ in the frame $t_1,\ldots,t_r$. 
For $p$ sufficiently large, we obtain the asymptotics of $Q_p(x)$ 
from the asymptotics of $P_p(x,x)$ stated in Theorem \ref{bkt2.1}. 
More precisely, we have the uniform expansion in any local $\cC^\ell$-norm near $x_0$,
\begin{equation}
P_p(x,x)=Q_p(x)^{-1}=p^n b_0(x) \left(\Id_E+O(p^{-1})\right),\text{ where } b_0={\det}(\dot{R}^L/(2\pi)).
\label{eq:3.12Oct24}
\end{equation}
As a consequence, we have the analogous expansion in any local $\cC^\ell$-norm 
\begin{equation}
Q_p(x)=p^{-n}(b_0(x))^{-1}\left(\Id_E+O(p^{-1})\right).
\label{eq:3.13Oct24}
\end{equation}
Using this we obtain 
\begin{equation}
\begin{split}
& \partial Q_p(x)=-p^{-n}\frac{\partial b_0(x) }{(b_0(x))^{2}}\,\Id_E+O(p^{-n-1}),\\
& \bar{\partial} Q_p(x)=-p^{-n}\frac{\bar{\partial}b_0(x) }{(b_0(x))^{2}}\,\Id_E+O(p^{-n-1}).
\end{split}
\label{eq:3.14Oct24}
\end{equation}

Consider the $r\times r$-matrix valued functions $H^{\Phi_{p}^\star(\cT^\star)}=\big(H^{\Phi_{p}^\star(\cT^\star)}_{jm}\big)$, 
$H^{L^{p}\otimes E}=\big(H^{L^{p}\otimes E}_{jm}\big)$ defined near $x_0$ by 
\[H^{\Phi_{p}^\star(\cT^\star)}_{jm}(x)=h^{\Phi_{p}^\star(\cT^\star)}(\tau_{m}(x),\tau_{j}(x))\,,\,\;
H^{L^{p}\otimes E}_{jm}(x)=h^{L^{p}\otimes E}(t_{m}(x),t_{j}(x)).\]
By \eqref{e:isostarP1} we obtain 
\begin{equation}
H^{\Phi_{p}^\star(\cT^\star)}=H^{L^{p}\otimes E}Q_p.
\label{eq:3.2.18}
\end{equation}
As in \eqref{eq:2.22New}-\eqref{eq:2.22New1}, the curvature tensor is given by
\[
		R^{\Phi_{p}^\star(\cT^\star)}=-(H^{\Phi_{p}^\star(\cT^\star)})^{-1}\bar{\partial} 
		H^{\Phi_{p}^\star(\cT^\star)}\wedge (H^{\Phi_{p}^\star(\cT^\star)})^{-1}\partial 
		H^{\Phi_{p}^\star(\cT^\star)}+(H^{\Phi_{p}^\star(\cT^\star)})^{-1}\bar{\partial}\partial 
		H^{\Phi_{p}^\star(\cT^\star)}.
	\]
Using \eqref{eq:3.2.18} we obtain
\begin{equation*}
\begin{split}
&R^{\Phi_{p}^\star(\cT^\star)}=\\&-\big(Q_{p}^{-1}(H^{L^{p}\otimes E})^{-1}
(\bar{\partial} H^{L^{p}\otimes E})Q_{p}+Q_{p}^{-1}\bar{\partial}Q_{p}\big)\wedge 
\big(Q_{p}^{-1}(H^{L^{p}\otimes E})^{-1}(\partial H^{L^{p}\otimes E})Q_{p}+
Q_{p}^{-1}\partial Q_{p}\big) \\
&+Q_{p}^{-1}(H^{L^{p}\otimes E})^{-1}(\bar{\partial}\partial H^{L^{p}\otimes E})
Q_{p}-Q_{p}^{-1}(H^{L^{p}\otimes E})^{-1}\partial H^{L^{p}\otimes E}
\wedge\bar{\partial}Q_{p}\\
&+Q_{p}^{-1}(H^{L^{p}\otimes E})^{-1}
\bar{\partial}H^{L^{p}\otimes E}
\wedge\partial Q_{p} +Q_{p}^{-1}\bar{\partial}\partial Q_{p}
\end{split}
\end{equation*}
Hence
\begin{equation}\label{e:id1}	
\begin{split}		
R^{\Phi_{p}^\star(\cT^\star)}&=Q_{p}^{-1}R^{L^{p}\otimes E}Q_{p}-Q_{p}^{-1}
\bar{\partial}Q_{p} \wedge Q_{p}^{-1}\partial 
Q_{p}+Q_{p}^{-1}\bar{\partial}\partial Q_{p}\\
&\hspace{5mm}-Q_{p}^{-1}(H^{L^{p}\otimes E})^{-1}
\partial H^{L^{p}\otimes 
			E}\wedge\bar{\partial}Q_{p} \\
			&\hspace{5mm}-Q_{p}^{-1}(\bar{\partial}Q_{p}) Q_{p}^{-1}\wedge 
			(H^{L^{p}\otimes E})^{-1}(\partial H^{L^{p}\otimes E})Q_{p}. 
\end{split}
\end{equation}

Set $H^{E}_{jm}=h^{E}(v_{m},v_{j})$, $\varphi_{L}(x)=|e_{L}|_{h^{L}}^{2}(x)$. Then
	\[H^{L^{p}\otimes E}=H^{E}\varphi_{L}^{p}.\]
A direct computation shows
\begin{equation}\label{eq:250305}
(H^{L^{p}\otimes E})^{-1}\partial H^{L^{p}\otimes E}= 
(H^{E})^{-1}\partial 
H^{E}+\frac{p}{\varphi_{L}}\,\partial\varphi_{L}\mathrm{Id}_{E}.
\end{equation}
The presence of $\mathrm{Id}_{E}$ in formula \eqref{eq:250305}
transforms \eqref{e:id1} into the following:
\begin{align*}
	R^{\Phi_{p}^\star(\cT^\star)}
	=Q_{p}^{-1}&R^{L^{p}\otimes E}Q_{p}-Q_{p}^{-1}\bar{\partial}Q_{p} \wedge Q_{p}^{-1}\partial 
	Q_{p}+Q_{p}^{-1}\bar{\partial}\partial Q_{p}\\
	&-Q_{p}^{-1}(H^{E})^{-1}\partial H^{E}\wedge\bar{\partial}Q_{p}-  
	Q_{p}^{-1}(\bar{\partial}Q_{p}) Q_{p}^{-1}\wedge(H^{E})^{-1}(\partial H^{E})Q_{p}.
\end{align*}
Hence
\begin{equation}\label{e:id2}
\Phi_p^\star\big(R^{\cT^\star}\big)=R^{\Phi_{p}^\star(\cT^\star)}=Q_p^{-1}(R^{L^{p}\otimes E}+A_p)Q_p,
\end{equation}
where 
\begin{align*}
A_p=-\bar{\partial}Q_{p}& \wedge Q_{p}^{-1}(\partial Q_{p})Q_{p}^{-1}+
(\bar{\partial}\partial Q_{p})Q_{p}^{-1}\\
&-(H^{E})^{-1}\partial H^{E}\wedge(\bar{\partial}Q_{p})Q_{p}^{-1}
-(\bar{\partial}Q_{p}) Q_{p}^{-1}\wedge(H^{E})^{-1}\partial H^{E}.
\end{align*}
Using \eqref{eq:3.12Oct24}--\eqref{eq:3.14Oct24} 
and the fact the leading terms in these 
asymptotic expansions are diagonal matrices, 
we infer that for every $\ell$ there exists $C_\ell >0$ such that 
for all $p\geq1$,  
\begin{equation}\label{e:Ap}
\left\|A_p\right\|_{\cC^\ell(\mathrm{near\;}x_{0})}\leq C_\ell.
\end{equation}
Since $R^{L^{p}\otimes E}=pR^{L}\otimes \mathrm{Id}_{E}+R^E$, 
it follows from \eqref{e:symmpol}, 
\eqref{e:Chern-m}, \eqref{e:Chern2} and \eqref{e:Ap}
that for every $\ell$ there exists $C'_\ell >0$ 
such that for all $p\geq1$ and $0\leq k\leq r$, we have
\begin{equation}\label{e:Tian-est}
\left\|c_k\Big(\frac{i}{2\pi}(R^{L^{p}\otimes E}+A_p)\Big)-c_{k}(L^{p}\otimes E, h^{L^{p}\otimes 
E})\right\|_{\cC^\ell(\mathrm{near\;}x_{0})}\leq C'_\ell\,p^{k-1}.
\end{equation}
Note that the above computations are local, but the estimates
hold true uniformly on the entire manifold $X$ due to the compactness of $X$. 
By \eqref{e:id2} we have that 
\begin{equation}\label{e:id3} 
\Phi_p^\star(c_k(\cT^\star,h^{\cT^\star}))=c_k\Big(\frac{i}{2\pi}(R^{L^{p}\otimes E}+A_p)\Big),
\end{equation}
and 
\eqref{e:Tian2} follows directly from \eqref{e:Tian-est} and \eqref{e:id3}.
\end{proof}

\begin{Remark} For $k=1$, \eqref{e:Tian} 
coincides with \cite[Theorem 5.1.17]{MM07},
\begin{equation}\label{eq:3.2.26}
\frac{1}{p}\,\Phi_{p}^\star(\omega_{\FS})-r c_1(L,h^L)=
O\left(\frac{1}{p}\right),
\end{equation}
where $\omega_{\FS}$ is the Fubini-Study form on 
$\mathbb{G}(r,H^{0}(X,L^{p}\otimes E)^\star)$.
\end{Remark}
We now consider the case when
$(L,h^L)$ is a prequantum line bundle of $(X,\omega)$, 
i.e., \eqref{eq:3.1.3} holds. 
In this case, if $(E,h^E)$ is the trivial line bundle, Tian's estimate 
\eqref{eq:3.2.26} of the speed of convergence of the induced Fubini-Study metric
to $\omega$ can be improved as shown in \cite[Remark 5.1.5]{MM07}:
\begin{equation}\label{eq:3.2.27}
\frac{1}{p}\,\Phi_{p}^\star(\omega_{\FS})-c_1(L,h^L)=
O\left(\frac{1}{p^2}\right).
\end{equation}
For general Chern forms we obtain
the following 
generalization of \eqref{eq:3.2.27}.
\begin{Theorem}\label{T:Tian3}
Let $(X,\omega)$ be a compact K\"ahler manifold of dimension $n$, 
$(L,h^L)$ be a positive line bundle on $X$, and $(E,h^E)$ be a Hermitian 
holomorphic vector bundle on $X$ of rank $r\leq n$. 
Assume that $\omega=c_{1}(L,h^{L})$. 
Then for every $\ell\in\N$ there exists $C_\ell>0$ such that 
\begin{equation}\label{e:3.2.28}
\big\|\Phi_p^\star(c_k(\cT^\star,h^{\cT^\star}))-
c_{k}(L^{p}\otimes E, h^{L^{p}\otimes E})\big\|_{\cC^\ell(X)}
\leq C_\ell\,p^{k-2}
\end{equation}
holds for $0\leq k\leq r$, and for all $p$ sufficiently large.
\end{Theorem}
\begin{proof}
Since $\omega=c_{1}(L,h^{L})$ we have that 
$b_0\equiv 1$ on $X$ (see \eqref{eq:3.12Oct24}). In the setting of the 
proof of Theorem \ref{T:Tian2} near $x_0\in X$, 
instead of \eqref{eq:3.14Oct24} we now have 
\[
\partial Q_p(x)=O(p^{-n-1}),\;
\bar{\partial} Q_p(x)= O(p^{-n-1}),\; 
\partial \bar{\partial} Q_p(x)= O(p^{-n-1}).
\]
So the upper bounds in \eqref{e:Ap} and respectively \eqref{e:Tian-est}
become
\[\left\|A_p\right\|_{\cC^\ell(\mathrm{near\;}x_{0})}\leq
\frac{C_\ell}{p}\;,\,\;
\left\|c_k\Big(\frac{i}{2\pi}(R^{L^{p}\otimes E}+A_p)\Big)-
c_{k}(L^{p}\otimes E, h^{L^{p}\otimes 
E})\right\|_{\cC^\ell(\mathrm{near\;}x_{0})}
\leq C'_\ell\,p^{k-2}.\]
This directly implies the estimate \eqref{e:3.2.28}.
\end{proof}
As a consequence of Theorem \ref{T:Tian3},
we compute the second-order approximation of the pullback of 
the $k$-th Chern form.
In the semiclassical limit, we retrieve the first Chern form 
$c_1(E,h^E)$ of $(E,h^E)$. This proves the semiclassical asymptotics 
\eqref{e:Tian1} and completes the proof of Theorem \ref{T:Tian}.
\begin{Corollary}
Let $(X,\omega)$ be a compact K\"ahler manifold of dimension $n$, 
$(L,h^L)$ be a positive line bundle on $X$, and $(E,h^E)$ be a Hermitian 
holomorphic vector bundle on $X$ of rank $r\leq n$. 
Assume that $\omega=c_{1}(L,h^{L})$. 
Then we have in the $\cC^\infty$ topology on $X$,
\begin{equation}
\Phi_p^\star(c_k(\cT^\star,h^{\cT^\star}))=
p^k\binom{r}{k}\omega^{k}
+p^{k-1}\binom{r-1}{k-1}c_1(E,h^E)\wedge \omega^{k-1}+
O(p^{k-2}).
\label{eq:3.28MAR25}
\end{equation}

\end{Corollary}
\begin{proof}
This follows directly from Theorem \ref{T:Tian3} and \eqref{eq:3.2.4bis}.
\end{proof}

Consequently, if $\omega=c_{1}(L,h^{L})$
we obtain the following refinement of 
\eqref{eq:3.2.26} and generalization of \eqref{eq:3.2.27},
\begin{equation}\label{eq:3.2.30}
\Phi_{p}^\star(\omega_{\FS})=pr c_1(L,h^L)
+c_1(E,h^E)+
O\left(\frac{1}{p}\right).
\end{equation}
If $(E,h^E)$ is trivial this gives 
\begin{equation}\label{eq:3.2.34}
\frac{1}{p}\,\Phi_{p}^\star(\omega_{\FS})-r c_1(L,h^L)=
O\left(\frac{1}{p^2}\right).
\end{equation}

\section{Kodaira maps as meromorphic transforms}\label{S:MTG}
In this section, we show that the Kodaira map can be interpreted as 
a meromorphic transform of codimension $n-r$ in the sense of Dinh-Sibony \cite{DS06}.
Furthermore, we establish that the intermediate degrees of this transform 
can be expressed in terms of the Chern classes of the vector bundle $E$.

Let $(X,\omega),\,(E,h^E),\,V,\,N$ verify assumptions (A)-(C), and $\Phi_E:X\to\G(r,V^\star)$ be the Kodaira map defined in \eqref{e:Kod2}.
It follows from (B) that $\dim Z_s\geq n-r$ for every $s\in\P V$. 
Moreover, Bertini's theorem for vector bundles \cite[Theorem 1.2]{MZ23} 
implies that $Z_s$ is a complex submanifold of $X$ of dimension $n-r$
for all $s$ outside a proper analytic subset of $\P V$. Hence the current 
of integration $[Z_s]=[s=0]$ along $Z_s$ is a well defined 
positive closed current of bidegree $(r,r)$ for such $s$. 
Locally, if $U\subset X$ is an open set such that $E|_U$
is trivial then $s$ is represented by an $r$-tuple of 
holomorphic functions $(f_1,\ldots , f_r)$ on $U$ and 
\[[s=0]=dd^c\log|f_1|\wedge\ldots\wedge dd^c\log|f_r|.\]
Consider the probability space $(\P V,\omega_\FS^N)$ 
and the current valued random variable $$\P V\ni s\longmapsto[Z_s].$$
The expectation of this random variable, also called the 
\textit{expected zero current} of the ensemble $(\P V,\omega_\FS^N)$, 
is the current $\E[Z_s]=\E_1(\omega_\FS^N)$ defined in \eqref{e:expk}. 
Namely,
\begin{equation}\label{e:exp1}
\langle \E[Z_s],\phi \rangle=\int_{\P V}\langle[s=0],\phi\rangle\,\omega_\FS^N\,,
\end{equation}
where $\phi$ is a smooth $(n-r,n-r)$ form on $X$. 
The following result of Sun \cite[Theorem 4.2]{Sun} provides a formula 
for $\mathbb{E}[Z_s]$ in terms of the $r$-th Chern form of 
$(\mathcal{T}^\star,h^{\mathcal{T}^\star})$.
A novel proof is presented in Section \ref{S:expectation}.
\begin{Theorem}[\cite{Sun}]\label{T:Sun}
If $(X,\omega),\,(E,h^E)$ verify (A)-(C) and $\Phi_E$ is defined in \eqref{e:Kod2} then
\[\E[Z_s]=\Phi_E^\star\big(c_r(\cT^\star, h^{\cT^\star})\big).\]
\end{Theorem}
We consider now the Kodaira map $\Phi_E$ as a meromorphic transform in the sense of Dinh-Sibony \cite{DS06}. Namely, it is the meromorphic transform $F_E:X\longrightarrow\P V$ with graph
\begin{equation}\label{e:MTGdef}
\Gamma_E=\{(x,s)\in X\times\P V:\,s(x)=0\}\subset X\times\P V.
\end{equation}
Let $\pi_1:X\times\P V\to X$, $\pi_2:X\times\P V\to\P V$ be the canonical projections.

\begin{Proposition}\label{P:MTG}
$F_E$ is a meromorphic transform of codimension $n-r$.
\end{Proposition}

\begin{proof} We have to check that the restrictions of the canonical projections to 
$\Gamma_E$ are surjective. It is clear by (B) that $\pi_2(\Gamma_E)=\P V$. By (C), 
the linear map $S\in V\to S(x)\in E_x$ has rank $r$ for every $x\in X$, so $\{s\in\P V:\,s(x)=0\}$ 
is a linear subspace of dimension $N-r$ of $\P V$. Hence $\pi_1(\Gamma_E)=X$ and 
$\Gamma_E$ is an irreducible analytic subset of $X\times\P V$ of dimension $N+n-r$.
\end{proof}

Let $I_2(F_E)=\{s\in\P V:\,\dim Z_s>n-r\}$ denote the second indeterminacy set of $F_E$. Then $\codim I_2(F_E)\geq2$. For $s\in\P V\setminus I_2(F_E)$, we denote by $[Z_s]=[s=0]$ the current of integration along $Z_s$. By \cite[Section 3.1]{DS06}, the pull-back of a current $T$ on $\P V$ of bidegree $(k,k)$, $N-r\leq k\leq N$,  is defined by 
\[F_E^\star(T)=(\pi_1)_\star\big(\pi_2^\star(T)\wedge[\Gamma_E]\big),\]
provided that $T$ is smooth, or $T$ is the current of integration over 
an analytic subset $H\subset\P V$ such that $\dim(H\cap I_2(F_E))\leq N-k-1$. 
Then $F_E^\star(T)$ is a current of bidegree $(k-N+r,k-N+r)$. 
Moreover, if $T$ is smooth then $F_E^\star(T)$ is given by a form 
with coefficients in $L^1(X)$. Similarly, if $S$ is a smooth $(k,k)$ 
form on $X$, where $n-r\leq k\leq n$, then the push-forward 
\[(F_E)_\star(S)=(\pi_2)_\star\big(\pi_1^\star(S)\wedge[\Gamma_E]\big)\]
is a well defined current of bidegree $(k-n+r,k-n+r)$ on $\P V$ and 
it is given by a form with coefficients in $L^1(\P V)$.
Let 
\begin{equation}\label{e:excep}
J=I_2(F_E)\cup\{s\in\P V:\,Z_s \text{ is not smooth}\}.
\end{equation}
By \cite[Theorem 1.2]{MZ23}, $J$ is an analytic subset of $\P V$. 
We infer that if $s\in\P V\setminus J$ then $Z_s$ is a complex submanifold 
of $X$ of dimension $n-r$ and 
\begin{equation}\label{e:Dirac}
F_E^\star(\delta_s)=[s=0].
\end{equation}
The intermediate degree of $F_E$ of order $k$, $N-r\leq k\leq N$, is defined by 
\begin{align}\label{e:intdeg}
\lambda_k(F_E)&=\int_XF_E^\star(\omega_\FS^k)\wedge\omega^{N+n-r-k}=
\int_{\P V}\omega_\FS^k\wedge(F_E)_\star(\omega^{N+n-r-k}) \\
&=\int_{\Gamma_E}\pi_1^\star(\omega^{N+n-r-k})
\wedge\pi_2^\star(\omega_\FS^k)\,, \nonumber
\end{align}
(see \cite[(3.1)]{DS06}). We show here that they can be expressed 
in terms of the Chern classes $c_{k-N+r}(E)$ of $E$.

\begin{Theorem}\label{T:MTGdeg}
Let $(X,\omega),\,(E,h^E)$ verify assumptions (A)-(C). 
If $F_E$ is the meromorphic transform defined in \eqref{e:MTGdef} then  
\begin{align}
\lambda_k(F_E)&=\int_Xc_{k-N+r}(E)\wedge\omega^{N+n-r-k},\;
N-r\leq k\leq N,\label{e:MTdeg} \\
F_E^\star(\omega_\FS^N)&=\Phi_E^\star\big(c_r(\cT^\star, h^{\cT^\star})\big).
\label{e:MTexp}
\end{align}
\end{Theorem}

\begin{proof}
Let $\phi$ be a smooth $(n-r,n-r)$ form on $X$. We have that 
\[\big\langle F_E^\star(\omega_\FS^N),\phi\big\rangle=
\int_{\Gamma_E}\pi_1^\star(\phi)\wedge\pi_2^\star(\omega_\FS^N)=
\int_{\P V}(F_E)_\star(\phi)\wedge\omega_\FS^N,\]
where $(F_E)_\star(\phi)\in L^1(\P V)$. If $s\not\in J$, where $J$ 
is defined in \eqref{e:excep}, then 
$\pi_2:\pi_2^{-1}(U)\cap\Gamma_E\to U$ 
is a submersion  for some neighborhood $U\subset\P V$ of $s$. We infer that  
\[(F_E)_\star(\phi)(s)=
\big(\pi_2|_{\Gamma_E}\big)_\star
\big(\big(\pi_1|_{\Gamma_E}\big)^\star\phi\big)(s)=
\int_{\pi_2^{-1}(s)\cap\Gamma_E}
\big(\pi_1|_{\Gamma_E}\big)^\star\phi=
\int_{Z_s}\phi=\langle[s=0],\phi\rangle,\;s\not\in J,\]
where the third equality follows since 
$\pi_1|_{\Gamma_E}:\pi_2^{-1}(s)\cap\Gamma_E=
Z_s\times\{s\}\to Z_s$ is a biholomorphism. This yields that 
\begin{equation}\label{e:MTexp1}
\big\langle F_E^\star(\omega_\FS^N),\phi\big\rangle=
\int_{\P V}\langle[s=0],\phi\rangle\,\omega_\FS^N=
\langle \E[Z_s],\phi \rangle,
\end{equation}
where $\E[Z_s]$ is the expectation current from \eqref{e:exp1}, 
and \eqref{e:MTexp} follows from Theorem \ref{T:Sun}. 
The Chern class $c_r(E)$ is the Poincar\'e dual of $Z_s$ 
for generic $s\in\P V$ (see e.g.\ \cite[p.\ 413]{GH94}), so 
\[\int_{Z_s}\omega^{n-r}=\int_Xc_r(E)\wedge\omega^{n-r}.\]
Hence \eqref{e:MTexp1} applied with $\phi=\omega^{n-r}$ 
yields \eqref{e:MTdeg} for $k=N$.

To prove \eqref{e:MTdeg} for $N-r\leq k\leq N$, 
we use a cohomological argument. Let $H_k\subset\P V$ 
be a linear subspace of dimension $N-k$ such that 
$\dim(H_k\cap I_2(F_E))\leq N-k-1$, and $s_0,\ldots,s_{N-k}\in\P V$ 
be linearly independent points such that 
\[H_k=\{a_0s_0+\ldots+a_{N-k}s_{N-k}:\,(a_0,\ldots,a_{N-k})
\in\C^{N-k+1}\setminus\{0\}\}.\]
If $[H_k]$ denotes the current of integration along $H_k$ 
then $\omega_\FS^k\sim[H_k]$, so by \eqref{e:intdeg} 
\[\lambda_k(F_E)=\int_XF_E^\star(\omega_\FS^k)\wedge\omega^{N+n-r-k}=
\int_XF_E^\star([H_k])\wedge\omega^{N+n-r-k}.\]

We have that $F_E^\star([H_k])$ is the current of integration 
along $F_E^{-1}(H_k)$, and $F_E^{-1}(H_k)$ is the degeneracy set 
$D_{N-k+1}(s_0,\ldots,s_{N-k})$, i.e.\ the set of points $x\in X$ 
where $s_0(x),\ldots,s_{N-k}(x)$ are linearly dependent. Indeed, 
\begin{equation}\label{e:degci}
F_E^{-1}(H_k)=\pi_1\big(\pi_2^{-1}(H_k)\cap\Gamma_E\big)=
\bigcup_{s\in H_k}Z_s=\{x\in X:\,(s_0\wedge\ldots\wedge s_{N-k})(x)=0\}.
\end{equation}
By \cite[p.\ 413]{GH94}, the Poincar\'e dual of 
$D_{N-k+1}(s_0,\ldots,s_{N-k})$ is $c_{k-N+r}(E)$. Therefore
\[\lambda_k(F_E)=\int_XF_E^\star([H_k])\wedge\omega^{N+n-r-k}=
\int_{F_E^{-1}(H_k)}\omega^{N+n-r-k}=
\int_Xc_{k-N+r}(E)\wedge\omega^{N+n-r-k}.\]
\end{proof}

\section{Expectation currents of degeneracy sets}\label{S:expectation}

In this section, we investigate the expectation currents $\E_k(\mu_k)$ 
defined in \eqref{e:expk} and we establish Theorems \ref{T:expcoh} and 
\ref{T:expdeg}.
When $k=1$, our result recover those obtained by Sun \cite{Sun} 
(as presented in Theorem \ref{T:Sun}), 
provided that the Kodaira map $\Phi_E$ is an embedding. 

\subsection{Cohomology class of expectation currents}\label{SS:exp}

In this subsection, we provide the proof of Theorem \ref{T:expcoh}.

\begin{proof}[Proof of Theorem \ref{T:expcoh}]
For $s=(s_1,\ldots,s_k)\in(\P V)^k$ let $H(s)\subset\P V$
be the linear subspace spanned by $s_1,\ldots,s_k$.
If $s$ is generic, we have that $H(s)$ has dimension $k-1$ and 
$\dim(H(s)\cap I_2(F_E))\leq k-2$, where $I_2(F_E)$
is the second indeterminacy set of the meromorphic transform
$F_E$ defined in \eqref{e:MTGdef}. By \eqref{e:degci}
we infer that $D_k(s_1,\ldots,s_k)=F_E^{-1}(H(s))$ and 
\[[s_1\wedge\ldots\wedge s_k=0]=F_E^\star([H(s)])\] 
is a positive closed current of bidegree $(r+1-k,r+1-k)$ on $X$. 
If $S_0,\ldots,S_N$ is a basis of $V$, we have a bijective map 
\[\P^N\ni[x_0:\ldots:x_N]\longmapsto x_0S_0+\ldots+x_NS_N\in\P V.\]
Let $U\subset X$ be an open set such that $E|_U$ 
has a holomorphic frame $e_1,\ldots,e_r$, and write 
\[S_j=\sum_{\ell=1}^rh_{j\ell}e_\ell, \text{ where } 
h_{j\ell}\in\cO_X(U).\]
If $s=(s_1,\ldots,s_k)\in(\P V)^k$ there exist 
$a_m=[a_{m0}:\ldots:a_{mN}]\in\P^N$ such that 
\[s_m=\sum_{j=0}^Na_{mj}S_j=
\sum_{\ell=1}^r\left(\sum_{j=0}^Na_{mj}h_{j\ell}\right)e_\ell,\;
1\leq m\leq k.\]
Therefore $(s_1\wedge\ldots\wedge s_k)(x)=0$ 
for some $x\in U$ if and only if $\rank A(s,x)\leq k-1$, where 
\[A(s,x)=A(a_1,\ldots,a_k;x)=
\left[\sum_{j=0}^Na_{mj}h_{j\ell}(x)\right]_{1\leq m\leq k,1\leq\ell\leq r}\] 
is a $k\times r$ matrix of holomorphic functions on $(\P^N)^k\times U$. 
Let $g_j$ be the determinant of the $k\times k$ minor of $A$ consisting 
of its first $k-1$ columns together with the column $j+k-1$, $1\leq j\leq r+1-k$. 
Then $g_j$ is a holomorphic function on $(\P^N)^k\times U$ 
and we may assume that for generic $s\in(\P V)^k\equiv(\P^N)^k$ we have 
\begin{equation}\label{e:degwedge}
[s_1\wedge\ldots\wedge s_k=0]=
dd^c\log|g_1(s,\cdot)|\wedge\ldots\wedge dd^c\log|g_{r+1-k}(s,\cdot)| 
\text{ on } U.
\end{equation}

Let $\theta$ be a fixed smooth $(n+k-r-1,n+k-r-1)$ 
form supported in $U$. Then 
\[f_\varepsilon(s):=\frac{1}{2^{r+1-k}}\int_Udd^c\log(|g_1(s,\cdot)|^2+
\varepsilon)\wedge\ldots\wedge dd^c\log(|g_{r+1-k}(s,\cdot)|^2+
\varepsilon)\wedge\theta,\;\varepsilon>0,\]
are smooth functions on $(\P V)^k$. Moreover, 
$f_\varepsilon(s)\to\langle[s_1\wedge\ldots\wedge s_k=0],\theta\rangle$ as 
$\varepsilon\searrow0$, for generic $s\in(\P V)^k$. 
Using this we deduce that the function 
\[f(s):=\langle[s_1\wedge\ldots\wedge s_k=0],\phi\rangle,\;
s=(s_1,\ldots,s_k)\in(\P V)^k,\] 
is Borel measurable, where $\phi$ is a fixed smooth $(n+k-r-1,n+k-r-1)$ 
form on $X$. We now show that $f$ is bounded. Indeed, we can assume 
that $\phi$ is real, so there exists $C_\phi>0$ such that 
$$-C_\phi\,\omega^{n+k-r-1}\leq\phi\leq C_\phi\,\omega^{n+k-r-1}.$$ 
Since the Poincar\'e dual of $D_k(s_1,\ldots,s_k)$ is 
$c_{r+1-k}(E)$ \cite[p.\ 413]{GH94}, we obtain, 
for generic $s\in(\P V)^k$, that 
\[|f(s)|\leq C_\phi\langle[s_1\wedge\ldots\wedge s_k=0],
\omega^{n+k-r-1}\rangle=
C_\phi\int_Xc_{r+1-k}(E)\wedge\omega^{n+k-r-1}.\]

We conclude that $\E_k(\mu_k)$ is a well-defined current of bidegree $(r+1-k,r+1-k)$ 
on $X$, and it is clearly positive and closed. 
Moreover, if $\phi$ is a closed smooth $(n+k-r-1,n+k-r-1)$ form on $X$ then 
\begin{align*}
\langle \E_k(\mu_k),\phi \rangle&=
\int_{(\P V)^k}\langle[s_1\wedge\ldots\wedge s_k=0],\phi\rangle\,d\mu_k\\
&=\int_{(\P V)^k}\left(\int_Xc_{r+1-k}(E)\wedge\phi\right)d\mu_k
=\int_Xc_{r+1-k}(E)\wedge\phi.
\end{align*}
This shows that $\E_k(\mu_k)$ belongs to the Chern class $c_{r+1-k}(E)$.
\end{proof}

\begin{Remark}\label{R:expec}
We can also consider the expectation current 
\[\E_k(\mu_k)=\int_{V^k}[s_1\wedge\ldots\wedge s_k=0]\,d\mu_k\]
with respect to a probability measure $\mu_k$ on $V^k$ 
that does not charge proper analytic subsets. 
Arguing as in the proof of Theorem \ref{T:expcoh} 
we see that $\E_k(\mu_k)$ is a well-defined positive closed current of bidegree 
$(r+1-k,r+1-k)$ on $X$ in the Chern class $c_{r+1-k}(E)$.
\end{Remark}

If $\mu_1$ is a continuous probability measure on $\P V$, 
we remark that the case $k=1$ of Theorem \ref{T:expcoh} 
follows from \cite[Theorem 4.2]{Sun}.

\subsection{Invariance properties under group actions}\label{SS:inv}

In this section we obtain the formula for the expectation currents 
of degeneracy sets in the case of the dual universal bundle on a 
Grassmannian (see Theorem \ref{T:expG} below).

Let $(X,\omega),\,(E,h^E),\,V,\,N$ verify (A)-(C) and $V$
be endowed with the $L^2$-inner product \eqref{e:ip2}. 
Assume that there exist holomorphic automorphisms $g\in\Aut(X)$ 
and $G\in\Aut(E)$ such that for every $x\in X$, $G:E|_x\to E|_{g(x)}$ 
is a linear isometry, i.e.
\begin{equation}\label{e:isom}
h^E_{g(x)}(G(u),G(v))=h^E_x(u,v),\;\forall\,u,v\in E_x.
\end{equation}
Then 
\begin{equation}\label{e:A}
A_1:V\to V\,,\,\;A_1(S)=G\circ S\circ g^{-1},
\end{equation}
is a linear isomorphism, and we define
\begin{equation}\label{e:Ak}
A_k:V^k\to V^k\,,\,\;A_k(S_1,\ldots,S_k)=(A_1(S_1),\ldots,A_1(S_k)).
\end{equation}
We denote again by $A_k$ the induced maps $A_k:(\P V)^k\to(\P V)^k$.

\begin{Lemma}\label{L:inv}
In the above setting, the following hold:

(i) $g^\star(c_k(E,h^E))=c_k(E,h^E)$, $0\leq k\leq r$.

(ii) If $(A_k)_\star\mu_k=\mu_k$ then $g_\star(\E_k(\mu_k))=\E_k(\mu_k)$.
\end{Lemma}

\begin{proof}
$(i)$ Let $e_1(x),\ldots,e_r(x)$ be a holomorphic frame for $E$
in a neighborhood of a point $x_0\in X$. Then $A_1(e_1),\ldots,A_1(e_r)$ 
is a holomorphic frame for $E$ near $g(x_0)$. Consider the matrices 
\[H_1=\Big[h^E_x(e_k,e_j)\Big]_{1\leq j\leq r,1\leq k\leq r}\,,\,\;
H_2=\Big[h^E_{g(x)}(A_1(e_k),A_1(e_j))\Big]_{1\leq j\leq r,1\leq k\leq r}\]
that represent $h^E$ near $x_0$, respectively near $g(x_0)$. 
By the isometry property \eqref{e:isom} of $G$ we infer that 
\[H_2=\Big[h^E_x(e_k\circ g^{-1},e_j\circ g^{-1})\Big]_{j,k},\text{ hence }
g^\star H_2=\Big[g^\star h^E_x(e_k\circ g^{-1},e_j\circ g^{-1})\Big]_{j,k}=H_1.\]
Therefore 
\[g^\star R^E_{g(x)}=g^\star\big(\db(H_2^{-1}\partial H_2)\big)=
\db\big((g^\star H_2)^{-1}\partial(g^\star H_2)\big)=
\db(H_1^{-1}\partial H_1)= R^E_{x},\]
which yields $(i)$.

$(ii)$ Let $\phi$ be a smooth $(n+k-r-1,n+k-r-1)$ form on $X$. Then
\[\langle g_\star(\E_k(\mu_k)),\phi\rangle=
\int_{(\P V)^k}\int_{\{s_1\wedge\ldots\wedge s_k=0\}}g^\star\phi\;d\mu_k=
\int_{(\P V)^k}\int_{\{g(x):\,s_1(x)\wedge\ldots\wedge s_k(x)=0\}}\phi\;d\mu_k.\]
Since $G:E|_x\to E|_{g(x)}$ is an isomorphism we have by \eqref{e:A}
\begin{align*}
\{g(x):\,s_1(x)\wedge\ldots\wedge s_k(x)=0\}&=
\{y:\,s_1(g^{-1}(y))\wedge\ldots\wedge s_k(g^{-1}(y))=0\}\\
&=\{y:\,A_1(s_1)(y)\wedge\ldots\wedge A_1(s_k)(y)=0\}.
\end{align*}
Hence by the invariance assumption on $\mu_k$ we obtain
\begin{align*}
\langle g_\star(\E_k(\mu_k)),\phi\rangle&=
\int_{(\P V)^k}\int_{\{A_1(s_1)\wedge\ldots\wedge A_1(s_k)=0\}}\phi\;d\mu_k\\
&=\int_{(\P V)^k}\int_{\{s_1\wedge\ldots\wedge s_k=0\}}\phi\;d\big((A_k)_\star\mu_k\big)=
\langle\E_k(\mu_k),\phi\rangle.
\end{align*}
\end{proof}
We now apply these invariance properties to the case 
of the universal bundle on a Grassmannian. 
Let $W$ be a complex vector space of dimension $m$ 
endowed with an inner product $\langle\cdot,\cdot\rangle$, 
and $\G(r,W)$ be the Grassmannian of $r$-planes in $W$. 
Let $\cT\to\G(r,W)$ be the universal bundle endowed with 
the Hermitian metric $h^{\cT}$ induced by the inner product on $W$, 
and let $(\cT^\star,h^{\cT^\star})$ be the dual holomorphic vector bundle. 
The unitary group ${\rm U}(W)$ of $(W,\langle\cdot,\cdot\rangle)$ 
acts on $\G(r,W)$ in the obvious way: if $g\in {\rm U}(W)$ and $\{H\}\in\G(r,W)$, 
where $H$ is an $r$-plane in $W$, then $g\{H\}=\{g(H)\}$. 
This action lifts to $\cT^\star$ as follows: if $\ell\in H^\star$ 
we define $G(\ell)=\ell\circ g^{-1}\in g(H)^\star$. 
Since $H^0(\G(r,W),\cT^\star)=W^\star$, the map $A_1$ 
defined in \eqref{e:A} is given by 
\begin{equation}\label{e:AG}
A_1:W^\star\to W^\star\,,\,\;A_1(\ell)=\ell\circ g^{-1}.
\end{equation}
We consider $W^\star$ with the inner product $\langle\cdot,\cdot\rangle$ 
induced by that on $W$. Let $\omega_\FS$ be the Fubini-Study form on 
$\P W^\star$ and $\nu$ be the Gaussian probability measure on 
$W^\star$ given in coordinates with respect to an orthonormal basis of $W^\star$.

\begin{Lemma}\label{L:invG}
If $g\in{\rm U}(W)$ then $G\in\Aut(\cT^\star)$ verifies \eqref{e:isom}. 
Moreover, the map $A_1$ in \eqref{e:AG} is an isometry of $W^\star$, 
hence $A_1^\star\omega_\FS=\omega_\FS$ and $(A_1)_\star\nu=\nu$. 
\end{Lemma}

\begin{proof}
For each $\ell\in W^\star$ there is a unique $w_\ell\in W$ 
such that $\ell(w)=\langle w,w_\ell\rangle$, $w\in W$. 
We have $\langle\ell,\ell'\rangle=\langle w_\ell,w_{\ell'}\rangle$ 
and if $g\in {\rm U}(W)$ then 
\[A_1(\ell)(w)=\langle g^{-1}(w),w_\ell\rangle=
\langle w,g(w_\ell)\rangle,\;
\langle A_1(\ell),A_1(\ell')\rangle=\langle g(w_\ell),g(w_{\ell'})\rangle=
\langle w_\ell,w_{\ell'}\rangle.\]
The assertion about $G$ follows in a similar way.
\end{proof}

\begin{Theorem}\label{T:expG}
In the above setting, let $\nu_k$ be the Gaussian probability measure 
on $(W^\star)^k$ and $\mu_k$ be the product measure on 
$(\P W^\star)^k$ determined by the Fubini-Study volume 
$(\omega_\FS)^{m-1}$ on $\P W^\star$. Then 
\[\E_k(\mu_k)=\E_k(\nu_k)=
c_{r+1-k}(\cT^\star,h^{\cT^\star}),\;1\leq k\leq r.\]
\end{Theorem}

\begin{proof}
We have $\mu_k=\omega_\FS(z_1)^{m-1}\wedge\ldots\wedge
\omega_\FS(z_k)^{m-1}$, 
where $\omega_\FS(z_j)=\pi_j^\star\omega_\FS$ 
and $\pi_j:(\P W^\star)^k\to\P W^\star_{z_j}$ 
is the projection on the $j$-th factor. 
If $A_k$ is the map defined in \eqref{e:Ak} then by 
Lemma \ref{L:invG} we infer that $(A_k)_\star\nu_k=
\nu_k$, $(A_k)_\star\mu_k=\mu_k$. 
Hence by Lemma \ref{L:inv} we have $g_\star\E_k=
\E_k$ for all $g\in {\rm U}(W)$, where $\E_k$ denotes either one of 
$\E_k(\mu_k)$ or $\E_k(\nu_k)$.

Let $dg$ denote the Haar measure on ${\rm U}(W)$. Then
\[\int_{{\rm U}(W)}(g^{-1})^\star\E_k\,dg=
\int_{{\rm U}(W)}g_\star\E_k\,dg=\E_k\,,\]
so the current $\E_k$ is represented by a smooth 
${\rm U}(W)$-invariant form \cite[Proposition 1.2.1]{Hu94}. 
Since 
\[\G(r,W)\cong{\rm U}(m)/{\rm U}(r)\times {\rm U}(m-r)\] 
is a symmetric space,  the ${\rm U}(W)$-invariant forms 
coincide with the harmonic forms (see e.g.\ \cite[p.\ 26]{LB13}), 
so each cohomology class contains a unique invariant form. 
By Lemma \ref{L:inv}, $c_{r+1-k}(\cT^\star,h^{\cT^\star})$ 
is ${\rm U}(W)$-invariant, and by Theorem \ref{T:expcoh}, 
$\E_k\in c_{r+1-k}(\cT^\star)$. 
Therefore $\E_k=c_{r+1-k}(\cT^\star,h^{\cT^\star})$.
\end{proof}

The case $k=1$ of Theorem \ref{T:expG} was obtained via 
different methods in \cite[Theorem 3.1]{Sun}.

\subsection{Formula for the expectation current}\label{SS:compexp}

We give here the proof of Theorem \ref{T:expdeg}. 
This will be done by using the Kodaira map \eqref{e:Kod2} 
and Theorem \ref{T:expG}. We will consider the case of 
the product measure $\mu_k$ on $(\P V)^k$. 
The proof in the case of the Gaussian probability measure on $V^k$ is identical. 
Let us write for simplicity 
\[Y=\G(r,V^\star),\;m=\dim Y=r(N+1-r),\;\Phi=\Phi_E:X\to Y,\]
where $\Phi_E$ is the Kodaira map defined in \eqref{e:Kod2}.
Recall that every $s\in V$ induces, by evaluation, 
a linear functional $\sigma_s:V^\star\to\C$, $\sigma_s(v^\star)=v^\star(s)$. 
By restriction to the fibers, $\sigma_s$ determines a global section 
$\sigma_s\in H^0(Y,\cT^\star)$, given by 
$\sigma_s(y)=\sigma_s|_{\cT_y}$, and we have a linear isometry
\begin{equation}
V\longrightarrow H^0(Y,\cT^\star),\;s\longmapsto\sigma_s.
\label{eq:isoGrass}
\end{equation}

Since $\Phi$ is holomorphic, we have as in \eqref{e:isostar} 
an isomorphism of holomorphic vector bundles on $X$,
\[\Psi:\Phi^\star(\cT^\star)\to E,\;\Psi((\Phi^\star\sigma_s)(x))=
s(x),\;s\in V,\;x\in X.\]
Moreover, if $s_1,\ldots,s_k\in\P V$ then 
\begin{equation}\label{e:pulldeg}
D_k(s_1,\ldots,s_k)=\Phi^{-1}\big(D_k(\sigma_{s_1},\ldots,
\sigma_{s_k})\big),\;1\leq k\leq r.
\end{equation}
By Theorem \ref{T:expcoh} and its proof we infer that for 
generic $(s_1,\ldots,s_k)\in(\P V)^k$ (i.e., outside an analytic subset 
of $(\P V)^k$) we have 
\begin{equation}\label{e:cod1}
\codim_X D_k(s_1,\ldots,s_k)=r+1-k,\;\codim_Y D_k(\sigma_{s_1},\ldots,
\sigma_{s_k})=r+1-k.
\end{equation}

\begin{Lemma}\label{L:push}
(i) If $(s_1,\ldots,s_k)\in(\P V)^k$ verifies \eqref{e:cod1} then 
\[\Phi_\star[D_k(s_1,\ldots,s_k)]=\big[D_k(\sigma_{s_1},\ldots,
\sigma_{s_k})\cap\Phi(X)\big].\]
(ii) We have 
\[\Phi_\star\E_k(\mu_k)=
\int_{(\P V)^k}\big[D_k(\sigma_{s_1},\ldots,
\sigma_{s_k})\cap\Phi(X)\big]\,d\mu_k(s_1,\ldots,s_k).\]
\end{Lemma}

\begin{proof} Let $\chi$ be a smooth $(n+k-r-1,n+k-r-1)$ 
form on $Y$ and set $Z=D_k(s_1,\ldots,s_k)$. 
Since $\Phi$ is an embedding, it follows by \eqref{e:pulldeg} 
and \eqref{e:cod1} that 
\begin{equation}\label{e:cod2}
\Phi(Z)=D_k(\sigma_{s_1},\ldots,\sigma_{s_k})\cap\Phi(X),\;
\codim_X Z=\codim_{\Phi(X)}\Phi(Z)=r+1-k.
\end{equation}
Hence $\Phi(Z)\subset\Phi(X)\subset Y$ is an analytic subset of 
pure dimension $n+k-r-1$. We have 
\[\langle\Phi_\star[D_k(s_1,\ldots,s_k)],\chi\rangle=
\int_Z\Phi^\star\chi=\int_{\Phi(Z)}\chi=
\langle\big[D_k(\sigma_{s_1},\ldots,
\sigma_{s_k})\cap\Phi(X)\big],\chi\rangle,\]
which proves $(i)$. Using $(i)$ we obtain
\begin{align*}
\langle\Phi_\star\E_k(\mu_k),\chi\rangle&=
\int_{(\P V)^k}\langle[D_k(s_1,\ldots,s_k)],
\Phi^\star\chi\rangle\,d\mu_k\\
&=\int_{(\P V)^k}\langle\big[D_k(\sigma_{s_1},\ldots,
\sigma_{s_k})\cap\Phi(X)\big],\chi\rangle\,d\mu_k.
\end{align*}
\end{proof}

\begin{Lemma}\label{L:ineq}
The current $T=\Phi^\star\big(c_{r+1-k}(\cT^\star,h^{\cT^\star})\big)-
\E_k(\mu_k)$ is strongly positive.
 \end{Lemma}

\begin{proof} The statement is local. Let $x_0\in X$. 
Since $\Phi$ is an embedding we can find 
a coordinate polydisc $U\subset Y$ centered at 
$y_0=\Phi(x_0)$ such that if 
$z=(z',z'')\in\C^{m-n}\times\C^n$ are the coordinates on $U$ 
then $\Phi(X)\cap U=\{z'=0\}$. 

Let $\chi$ be a smooth (weakly) positive $(n+k-r-1,n+k-r-1)$ form 
supported in $\Phi^{-1}(U)$. 
It is easy to see that there exists a smooth positive 
$(n+k-r-1,n+k-r-1)$ form $\wi\chi$ 
supported in $U$ such that $\Phi^\star\wi\chi=\chi$.

Assume that $(s_1,\ldots,s_k)\in(\P V)^k$ verifies \eqref{e:cod1}. 
As in the proof of Theorem \ref{T:expcoh}, we can write 
\begin{equation}\label{e:w1}
[D_k(\sigma_{s_1},\ldots,\sigma_{s_k})]=dd^c\log|g_1|\wedge\ldots
\wedge dd^c\log|g_{r+1-k}|,
\end{equation}
where $g_j$ are holomorphic functions on $U$ such that 
$\codim\{g_{j_1}=\ldots=g_{j_\ell}=0\}\geq\ell$ for each 
$1\leq j_1<\ldots<j_\ell\leq r+1-k$, so the wedge product 
if well defined by \cite[Corollary 2.11]{D93}, \cite[Corollary 3.6]{FS95}. 
Moreover, we have 
\begin{equation}\label{e:w2}
[\Phi(X)]=(dd^cv)^{m-n}, \text{ where } v(z)=\log\|z'\|,\;z\in U.
\end{equation}
Using again the results of \cite{D93,FS95}, we infer by \eqref{e:cod2}, 
\eqref{e:w1}, \eqref{e:w2} that the following wedge product of 
bidegree $(1,1)$ currents is well defined and 
\begin{equation}\label{e:w3}
dd^c\log|g_1|\wedge\ldots\wedge dd^c\log|g_{r+1-k}|\wedge(dd^cv)^{m-n}=
\big[D_k(\sigma_{s_1},\ldots,\sigma_{s_k})\cap\Phi(X)\big].
\end{equation}
Indeed, if $\theta$ is a smooth $(n+k-r-1,n+k-r-1)$ form supported in $U$ then 
\begin{align*}
\int_Udd^c\log&(|g_1|^2+\varepsilon)^{1/2}\wedge\ldots
\wedge dd^c\log(|g_{r+1-k}|^2+\varepsilon)^{1/2}
\wedge(dd^cv)^{m-n}\wedge\theta=\\
&\int_{\Phi(X)\cap U}\Big(\bigwedge_{j=1}^{r+1-k}dd^c\log(|g_j(0,z'')|^2+
\varepsilon)^{1/2}\Big)\wedge\theta(0,z'')\longrightarrow\\
&\int_{\Phi(X)\cap U}dd^c\log|g_1(0,z'')|\wedge\ldots\wedge 
dd^c\log|g_{r+1-k}(0,z'')|\wedge\theta(0,z'')=
\int_A\theta
\end{align*}
as $\varepsilon\searrow0$, where $\theta(0,z'')$ denotes the pullback of 
$\theta$ to $\Phi(X)\cap U$ and $A$ denotes the set of regular points of 
$D_k(\sigma_{s_1},\ldots,\sigma_{s_k})\cap\Phi(X)$.

Let $v_\varepsilon(z)=\frac{1}{2}\log(\|z'\|^2+\varepsilon)$. Then 
\begin{equation}\label{e:w4}
dd^c\log|g_1|\wedge\ldots\wedge dd^c\log|g_{r+1-k}|\wedge
(dd^cv_\varepsilon)^{m-n}\to 
\big[D_k(\sigma_{s_1},\ldots,\sigma_{s_k})\cap\Phi(X)\big] 
\end{equation}
as $\varepsilon\searrow0$, in the weak sense of currents. 
Using Lemma \ref{L:push}, \eqref{e:w4}, Fatou's lemma and \eqref{e:w1} we infer that 
\begin{align*}
\langle\E_k(\mu_k),\chi\rangle&=\langle\Phi_\star\E_k(\mu_k),\wi\chi\rangle=
\int_{(\P V)^k}\langle\big[D_k(\sigma_{s_1},\ldots,\sigma_{s_k})\cap\Phi(X)\big],\wi\chi\rangle\,d\mu_k\\
&\leq\liminf_{\varepsilon\searrow0}\int_{(\P V)^k}\int_Ydd^c\log|g_1|\wedge\ldots\wedge 
dd^c\log|g_{r+1-k}|\wedge(dd^cv_\varepsilon)^{m-n}\wedge\wi\chi\,d\mu_k\\
&=\liminf_{\varepsilon\searrow0}\int_Y\Big(\int_{(\P V)^k}[D_k(\sigma_{s_1},\ldots,\sigma_{s_k})]\,
d\mu_k\Big)\wedge(dd^cv_\varepsilon)^{m-n}\wedge\wi\chi.
\end{align*}
By Theorem \ref{T:expG} and \eqref{e:w2} we have, as $\varepsilon\searrow0$,
\begin{align*}
\int_Y\Big(\int_{(\P V)^k}[D_k&(\sigma_{s_1},\ldots,\sigma_{s_k})]\,
d\mu_k\Big)\wedge(dd^cv_\varepsilon)^{m-n}\wedge\wi\chi= \\ 
&\int_Yc_{r+1-k}(\cT^\star,h^{\cT^\star})\wedge(dd^cv_\varepsilon)^{m-n}\wedge\wi\chi\longrightarrow \\
&\int_{\Phi(X)}c_{r+1-k}(\cT^\star,h^{\cT^\star})\wedge\wi\chi=
\int_X\Phi^\star\big(c_{r+1-k}(\cT^\star,h^{\cT^\star})\wedge\chi.
\end{align*}
This concludes the proof of the lemma.
\end{proof}

\begin{proof}[Proof of Theorem \ref{T:expdeg}]
By Lemma \ref{L:ineq}, 
$T=\Phi^\star\big(c_{r+1-k}(\cT^\star,h^{\cT^\star})\big)-\E_k(\mu_k)$ 
is a positive closed current of bidegree $(r+1-k,r+1-k)$ on $X$ 
which is null-cohomologous. Indeed,
by Theorem \ref{T:expcoh} $\E_k(\mu_k)$ belongs to the Chern class 
$c_{r+1-k}(E)$, and so does the form 
$\Phi^\star\big(c_{r+1-k}(\cT^\star,h^{\cT^\star})\big)$ 
since the vector bundles $\Phi^\star(\cT^\star)$ and $E$ are isomorphic. 
Therefore $\int_X T\wedge\omega^{n+k-r-1}=0$, hence $T=0$. 
\end{proof}

\begin{proof}[Proof of Corollary \ref{cor:1.5}]
 Under the hypothesis of Corollary \ref{cor:1.5}, for all sufficiently large $p$, the Kodaira map $\Phi_p$ in \eqref{e:Kod1} is a holomorphic embedding. Moreover, by the isomorphism in \eqref{e:isostar}, assumptions (A)-(C) hold for $(X,\omega)$ and $(L^p\otimes E, h^{L^p\otimes E})$. So Theorem \ref{T:expdeg} applies, and we have 
 \[\E_k(\mu^p_k)=\E_k(\nu^p_k)=
\Phi_p^\star\big(c_{r+1-k}(\cT^\star,h^{\cT^\star})\big),\;1\leq k\leq r,\; p \gg 1.\]
Then \eqref{eq:cor1.5} and \eqref{eq:cor1.5-2} follow directly from \eqref{e:Tian}, respectively \eqref{e:Tian1}.
\end{proof}
\subsection{Degeneracy set of a tuple of sections}
\label{S:tuple}
Let us consider a special case of Theorem \ref{T:expdeg} 
for $E=L^{\oplus r}$ ($1\leq r\leq n$), where $(L,h^L)$ is a positive line bundle 
on $X$ such that the Kodaira map $\Phi_L: X\rightarrow \P V^\star_L $ 
is an embedding, where $V_L:=H^0(X,L)$. 
We also assume that every $r$ sections in $V_L$ have at least one common zero in $X$. 
Then $(X,\omega)$, $(E, h^E):=(L^{\oplus r}, (h^L)^{\oplus r})$ 
satisfy assumptions (A)-(C). Note that 
$V:=H^0(X, L^{\oplus r})=V_L^{\oplus r}$, and let $\nu_L$ 
be the Gaussian probability measure on $V_L$.

We define a canonical embedding 
\[\Theta_r: \P V_L^\star\rightarrow \mathbb{G}(r, V^\star) \]
that sends the complex line $[\xi]$ to the $r$-subspace $(\C\xi)^{\oplus r}$. 
Moreover, we have the isomorphism of the Hermitian bundles on $\P V_L^\star$:
\[\Theta_r^\star \cT^\star\simeq \mathcal{O}(1)^{\oplus r}.\]
As a consequence, on $\P V_L^\star $, for $0\leq k\leq r$,
\begin{equation}
    \Theta_r^\star c_k(\cT^\star,h^{\cT^\star})=
    \binom{r}{k}\omega_{\mathrm{FS}}^k,
\end{equation}
where $\omega_{\mathrm{FS}}$ is the Fubini-Study form on 
$\P V_L^\star$ associated to the $L^2$-inner product on $V_L$.

Let $\Phi_E: X\rightarrow \mathbb{G}(r, V^\star)$ 
be the Kodaira map \eqref{e:Kod1}. 
Then we have $\Phi_E=\Theta_r\circ \Phi_L.$
Hence $\Phi_E$ is an embedding, and
\begin{equation}
\Phi_E^\star(c_k(\cT^\star,h^{\cT^\star}))=
\binom{r}{k}\Phi_L^\star(\omega_{\mathrm{FS}})^k.
\label{eq:5.15new}
\end{equation}
When $k=r$, combining Theorem \ref{T:expdeg} and \eqref{eq:5.15new}, we recover \cite[Lemma 4.3]{ShZ99} for the expectation current of simultaneous zeros of $r$ independent random holomorphic sections of $L$.

 Consider the random square matrix 
 $(s_{j\ell})_{1\leq j,\ell\leq r}\in M_{r\times r}(V_L)\simeq V_L^{r^2}$, 
 where the coefficients $s_{j\ell}$ are independent and identically 
 distributed random elements in $V_L$ with the law $\nu_L$. 
 Define the determinant section
\begin{equation}\label{e:detsec}
    \det (s_{j\ell})_{1\leq j,\ell\leq r}\in H^0(X,L^r).
\end{equation}
Then the zero set of $\det(s_{j\ell})_{1\leq j,\ell\leq r}$
equals the degeneracy set of $(s_1,\ldots, s_r)\in V^r$, 
where $s_\ell=(s_{j\ell})_{1\leq j \leq r}\in V$. 
We endow on $V^r$ with the Gaussian probability measure $\nu_r=\nu_L^{r^2}$.
By Theorem \ref{T:expdeg} and \eqref{eq:5.15new}, 
the expectation current for the zeros of the determinant section \eqref{e:detsec} is given by
\begin{equation}
\E[ \det (s_{j\ell})_{1\leq j,\ell\leq r}=0]=
r\Phi_L^\star(\omega_{\mathrm{FS}}).
\label{eq:5.17new}
\end{equation}
We will discuss further the case $E=L^{\oplus r}$ in Section \ref{S:tuples}.


\section{Distribution of zeros of random sections}\label{S:zeros}

In this section we give the proof of Theorem \ref{T:zeros}. 
To this end we will use Theorem \ref{T:Tian}, 
together with an equidistribution theorem of Dinh and Sibony 
for meromorphic transforms \cite{DS06}. We will apply their result 
to the meromorphic transforms determined by the Kodaira maps into Grassmannians as in Section \ref{S:MTG}.

We start by recalling a few notions that we will need. 
Let $\P^\ell$ be the complex projective space of dimension $\ell$ 
and $\omega_\FS$ be the Fubini-Study form. We denote by $PSH(\P^\ell,\omega_\FS)$ 
the class of $\omega_\FS$-plurisubharmonic functions. 
These are the functions $\varphi$ on $\P^\ell$ which are locally the sum of a 
plurisubharmonic (psh) function and a smooth one, such that 
$\omega_\FS+dd^c\varphi\geq0$ in the sense of currents. Let 
\begin{align}
R(\P^\ell)&=\sup_\varphi\Big\{\max_{\P^\ell}\varphi:\,\varphi\in PSH(\P^\ell,\omega_\FS),\,
\int_{\P^\ell}\varphi\,\omega_\FS^\ell=0\Big\}, \label{e:R}\\
\Delta(\P^\ell,t)&= \sup_\varphi\Big\{\int_{\{\varphi<-t\}}\omega_\FS^\ell:\,
\varphi\in PSH(\P^\ell,\omega_\FS),\,\int_{\P^\ell}\varphi\,\omega_\FS^\ell=0\Big\},\;t\in\R. 
\label{e:Delta}
\end{align}
Then 
\begin{equation}\label{e:RDest}
R(\P^\ell)\leq\frac{1}{2}\,(1+\log\ell)\,,\,\;\Delta(\P^\ell,t)\leq 
c\ell e^{-\alpha t},\;\forall\,t\in\R,\end{equation}
where $c>0,\alpha>0$ are constants independent of 
$\ell$, cf.\ \cite[Proposition A.3,\ Corollary A.5]{DS06}.

\begin{Theorem}\label{T:zeros2}
Let $(X,\omega)$ be a compact K\"ahler manifold of dimension $n$, 
$(L,h^L)$ be a positive line bundle on $X$, and $(E,h^E)$ 
be a Hermitian holomorphic vector bundle of rank $r\leq n$ on $X$. 
Let $(\mathcal H,\Upsilon)$ be the probability space defined in \eqref{e:prob}. 
Then there exist $C>0$ and $p_0\in\N$ such that the following holds: 
For any $\gamma>1$ there exist subsets 
$\mathcal{E}_p=\mathcal{E}_p(\gamma)\subset\P V_p$ such that for $p>p_0$ we have 

(i) $\Upsilon_p(\mathcal{E}_p)\leq Cp^{-\gamma}$;

(ii) if $s_p\in\P V_p\setminus \mathcal{E}_p$ then
\[\Big|\frac{1}{p^r}\,\Big\langle[s_p=0]-\Phi_p^\star(c_r(\cT^\star,h^{\cT^\star})),
\phi\Big\rangle\Big|\leq C\gamma\,\frac{\log p}{p}\,\|\phi\|_{\cC^2(X)},\]
for any $(n-r,n-r)$ form $\phi$ of class $\cC^2$ on $X$.
Moreover, the estimate (ii) holds for $\Upsilon$-a.e. sequence
 $\{s_p\}_{p\geq1}\in\mathcal H$ provided
that $p$ is large enough.
\end{Theorem}

\begin{proof}
We consider the Kodaira maps $\Phi_p$ from \eqref{e:Kod1} 
as meromorphic transforms as in Section \ref{S:MTG} (see\eqref{e:Kod2},\eqref{e:MTGdef}). 
Since $(L,h^L)$ is positive, we have by \cite[Theorem 5.15]{MM07} 
that the vector bundles $L^p\otimes E$ verify (C) for all $p$ sufficiently large. 

Let $\Gamma_p\subset X\times\P V_p$ be the set 
defined by $\Phi_p$ as in \eqref{e:MTGdef}, 
where $\P V_p$ is endowed with the Fubini-Study form $\omega_\FS$. 
It follows from \cite[Corollary 3.6]{Kl69} that for all $p$ sufficiently large and 
for generic $s\in V_p$ we have $Z_{s}\neq\varnothing$, 
so the image of the projection $\pi_2(\Gamma_p)\subset\P V_p$ is dense. 
Thus $\pi_2$ is surjective since $\pi_2(\Gamma_p)$ is an analytic subset of $\P V_p$, 
so (B) holds as well. By Proposition \ref{P:MTG}, $\Phi_p$ 
defines a meromorphic transform $F_p$ of codimension $n-r$, 
with graph $\Gamma_p$. Let $d(F_p)$, $\delta(F_p)$ denote its intermediate degrees 
of order $d_p$, respectively $d_p-1$, see \eqref{e:intdeg}. 
Using Theorem \ref{T:MTGdeg} and Lemma \ref{L:ChernLp} we get 
\begin{align*}
d(F_p)&=\int_Xc_r(L^{p}\otimes E, h^{L^{p}\otimes E})\wedge\omega^{n-r}=
p^r\int_Xc_1(L,h^L)^r\wedge\omega^{n-r}+O(p^{r-1}),\\
\delta(F_p)&=\int_Xc_{r-1}(L^{p}\otimes E, h^{L^{p}\otimes E})\wedge\omega^{n-r+1}=
rp^{r-1}\int_Xc_1(L,h^L)^{r-1}\wedge\omega^{n-r+1}+O(p^{r-2}).
\end{align*}
Using these, \eqref{abk2.3} and \eqref{e:RDest} 
we infer that there exist $C_1>1$ and $p_0\in\N$
such that for all $p>p_0$ we have
\begin{equation}\label{e:est1}
\begin{split}
&d_p\leq C_1p^n\,,\,\;d(F_p)\leq C_1p^r\,,\,\;\frac{d(F_p)}{\delta(F_p)}\geq
\frac{p}{C_1},\\
&R(\P V_p)\leq C_1+\frac{n}{2}\log p\,,\,\;\Delta(\P V_p,t)\leq 
C_1p^ne^{-\alpha t},\;\forall\,t\in\R,
\end{split}
\end{equation}
where $R(\P V_p)$, $\Delta(\P V_p,t)$ are defined in \eqref{e:R} and 
\eqref{e:Delta}, and $\alpha>0$ is the constant from \eqref{e:RDest}.
For $\varepsilon>0$ let
\[\mathcal{E}'_p(\varepsilon):=\bigcup_{\|\phi\|_{\cC^2(X)}\leq 1}
\left\{s_p\in\P V_p:\,\left|\left\langle F_p^\star(\delta_{s_p}) - 
F_p^\star(\Upsilon_p),\phi\right\rangle\right| 
\geq d(F_p)\varepsilon \right\},\]
\[\mathcal{E}''_p(\varepsilon):=\bigcup_{\|\phi\|_{\cC^2(X)}\leq 1}
\left\{s_p\in\P V_p:\,\left|\left\langle F_p^\star(\delta_{s_p}) - 
F_p^\star(\Upsilon_p),\phi\right\rangle\right| 
\geq C_1p^r\varepsilon \right\},\]
where $\phi$ is a $(n-r,n-r)$ form on $X$ of class $\cC^2$. 
Using \eqref{e:est1} 
and \cite[Lemma 4.2\,(d)]{DS06} we obtain, for all $\varepsilon>0$,
\begin{equation}\label{e:est2}
\Upsilon_p(\mathcal{E}''_p(\varepsilon))\leq
\Upsilon_p(\mathcal{E}'_p(\varepsilon))\leq
\Delta(\P V_p,t_{\varepsilon,p}), \text{ where } t_{\varepsilon,p}=
\varepsilon\,\frac{d(F_p)}{\delta(F_p)}-3R(\P V_p).
\end{equation}

Recall by \eqref{e:Dirac} that $F_p^\star(\delta_{s_p})=[s_p=0]$ 
for $s_p$ outside an analytic subset of $\P V_p$. 
Moreover, by Theorem \ref{T:MTGdeg}, $F_p^\star(\Upsilon_p)=
\Phi_p^\star\big(c_r(\cT^\star, h^{\cT^\star})\big)$. 
Using \eqref{e:est1} 
and \eqref{e:est2} we get, for all $p>p_0$ and $\varepsilon>0$, that
\[t_{\varepsilon,p}\geq\frac{\varepsilon p}{C_1}-
3\Big(C_1+\frac{n}{2}\log p\Big)\,,\,\;
\Upsilon_p(\mathcal{E}''_p(\varepsilon))\leq 
C_1p^n\exp\Big(-\frac{\alpha\varepsilon p}{C_1}+3\alpha C_1+
\frac{3\alpha n}{2}\log p\Big).\]
We now take $\gamma>1$ and choose 
$\varepsilon=\varepsilon_{p,\gamma}$ such that
\[-\frac{\alpha p\varepsilon}{C_1}+3\alpha C_1+
\frac{3\alpha n}{2}\log p=-(n+\gamma)\log p.\]
Let $\mathcal{E}_p=\mathcal{E}_p(\gamma):=
\mathcal{E}''_p(\varepsilon_{p,\gamma})$. 
Then for all $p>p_0$ we have that
\[\Upsilon_p(\mathcal{E}_p(\gamma))\leq C_1p^{-\gamma}.\]
Moreover, if $s_p\in\P V_p\setminus \mathcal{E}_p$ and 
$\phi$ is a $(n-r,n-r)$ form on $X$ 
of class $\cC^2$, we infer by the definition of $\mathcal{E}_p$ that 
\[\Big|\frac{1}{p^r}\,\Big\langle[s_p=0]-
\Phi_p^\star(c_r(\cT^\star,h^{\cT^\star})),\phi\Big\rangle\Big|
\leq C_1\varepsilon_{p,\gamma}\|\phi\|_{\cC^2(X)}.\]
Note that 
\[\varepsilon_{p,\gamma}=\frac{C_1}{\alpha p}
\left(\Big(n+\gamma+\frac{3\alpha n}{2}\Big)\log p+
3\alpha C_1\right)\leq C_2\gamma\,\frac{\log p}{p}\]
holds for $p>p_0$, where $C_2>0$ is a constant independent 
of $p$ and $\gamma$. 
This yields assertion $(ii)$ of Theorem \ref{T:zeros2}.
Finally, the last assertion of Theorem \ref{T:zeros2} 
follows from the 
Borel-Cantelli lemma since $\sum_{p=1}^\infty\Upsilon_p(\mathcal{E}_p)<\infty$.
\end{proof}

\begin{proof}[Proof of Theorem \ref{T:zeros}]
The proof is a direct application of Theorems \ref{T:zeros2} and \ref{T:Tian}.
\end{proof}

\begin{Remark}\label{Rk:6.2}
Note that in Theorem \ref{T:zeros}, we may always replace $(\P V_p,\Upsilon_p)$ by $(V_p, \mathcal{G}_p)$, where $\mathcal{G}_p$ denotes the Gaussian probability measure on $V_p$ induced by the $L^2$-inner product. Indeed, if $\pi:V_p\setminus\{0\}\to\P V_p$ is the canonical projection and $\wi{\mathcal{E}_p}=\pi^{-1}(\mathcal{E}_p)$ then $\mathcal{G}_p(\wi{\mathcal{E}_p})=\Upsilon_p(\mathcal{E}_p)$, so Theorem \ref{T:zeros} holds for the sets $\wi{\mathcal{E}_p}$.

\end{Remark}

\section{Simultaneous zero sets of holomorphic sections}
\label{S:tuples}
Let $(X,\omega)$ be a compact K\"{a}hler manifold of dimension $n$  
and let $(L,h^L)$ be a positive line bundle on $X$. 
Let $E_0=X\times\C^r$ be the trivial 
vector bundle on $X$, where $r\leq n$.
In this section we specialize the results of  
Theorem \ref{T:zeros} and Corollary \ref{cor:1.5} to the case $E=E_0$. 
In particular, we recover the results of 
Shiffman-Zelditch \cite{ShZ99} and Dinh-Sibony \cite{DS06} 
for the simultaneous zeros of several independent random 
sections of a line bundle. 

If $V_p:=H^0(X,L^p\otimes E_0)$ we have a natural isomorphism
\begin{equation}
    V_p\cong
    H^0(X,L^p)^{\oplus r},
    \label{eq:7.1}
\end{equation}
which identifies sections of $L^p\otimes E_0$ to
$r$-tuples of sections of $L^p$. For $j=1,\ldots, r$, let 
$\Pi^p_j:V_p\rightarrow H^0(X,L^p) $ 
denote the projection onto the $j$-th component in the direct sum of \eqref{eq:7.1}. 
Let $h^{E_0}_0$ denote the trivial Hermitian metric on $E_0$. 
Then, \eqref{eq:7.1} is also an isometry of the $L^2$-inner products induced 
by $h^L$ and $h^{E_0}_0$.
If we equip the trivial bundle $E_0$ with an arbitrary smooth Hermitian metric 
$h^{E_0}$, then the identification \eqref{eq:7.1} is generally no longer isometric, 
and the direct sum is not orthogonal with respect to the $L^2$-inner product induced 
by $h^L$ and $h^{E_0}$. 

Let $M_+(\C^r)$ be the set of positive definite Hermitian matrices of rank $r$, 
which carries the usual smooth structure. Then the set of smooth Hermitian metrics 
$h^{E_0}$ on $E_0$ is given by the set $\cC^\infty(X, M_+(\C^r))$
of all smooth functions on $X$ with values in $M_+(\C^r)$.

If $h^{E_0}_H$ on $E_0$ is given by an element 
$H\in\mathscr{C}^\infty(X, M_+(\C^r))$, that is, 
\[h^{E_0}_H(\cdot, \cdot)(x)=h^{E_0}_0(H(x)\cdot,\cdot),\;x\in X,\]
then the Chern curvature associated with $h^{E_0}_H$ 
is given by the following matrix of $(1,1)$-forms (see \eqref{eq:2.22New})
\begin{equation}
 R^{E}=H^{-1}\bar{\partial}\partial H-H^{-1}\bar{\partial}H\wedge H^{-1}\partial H.
\end{equation}
We can write the formula for $c_1(E_0,h^{E_0}_H)$ 
in terms of the matrix $H$,
\[c_1(E_0,h^{E_0}_H)=\frac{i}{2\pi}\Tr[H^{-1}\bar{\partial}\partial H-
H^{-1}\bar{\partial}H\wedge H^{-1}\partial H]=
-\frac{i}{2\pi} \partial\bar{\partial}\log \det H(x) \in \Omega^{(1,1)}(X).\]

We equip $V_p$ with the inner product induced by $h^L$ and $h^{E_0}_H$, 
and we denote by $\mathcal{G}_p^k=\mathcal{G}_{p,H}^k$ the corresponding 
Gaussian probability measure on $V_p^k$, where $1\leq k\leq r$ (it corresponds 
to the measure $\nu^p_k$ in Corollary \ref{cor:1.5}). 
If $(s^1_p, \ldots, s^k_p)\in V_p^k$ then $s^1_p \wedge \ldots \wedge s^k_p$ 
defines a random holomorphic section of $L^{pk} \otimes \Lambda^k E_0$ over $X$, 
where $\Lambda^k E_0=X\times\Lambda^k \mathbb{C}^r$. 
Let 
\[\E_k(\mathcal{G}_p^k)=\E[D_k(s^1_p,\ldots, s^k_p)]\]
be the expectation current of the degeneracy set of $s^1_p,\ldots, s^k_p$ defined 
in \eqref{e:expk}. Then Theorem \ref{T:zeros} and Corollary \ref{cor:1.5} hold in this 
setting (see also Remark \ref{Rk:6.2}). In particular, we have:

\begin{Proposition}\label{P:r-tuple}
In the above setting, where $h^{E_0}_H$ is the Hermitian metric associated to 
$H\in \mathscr{C}^\infty(X, M_+(\C^r))$, the following hold:

(i) If $s_p=(s_{p,1},\ldots, s_{p,r})\in V_p$ then, as $p\rightarrow\infty\,$, the normalized 
current of integration 
\[p^{-r}[s_p=0]=p^{-r}[s_{p,1}=\ldots =s_{p,r}=0]\] 
over the simultaneous zeros converges weakly almost surely to 
$c_1(L,h^L)^r$, with convergence speed $O(p^{-1}\log p)$.

(ii) If $\omega=c_1(L,h^L)$ then, as $p\to\infty$,
\begin{equation}\label{eq:7.2-2}
\frac{1}{p^{r+1-k}}\,\E_k(\mathcal{G}_p^k)=
\binom{r}{k-1} \omega^{r+1-k}-
\frac{i}{2\pi p}\binom{r-1}{r-k}\partial\bar{\partial}\log \det H\wedge 
\omega^{r-k}+O\left(\frac{1}{p^2}\right).
\end{equation}
The sub-leading term (coefficient of $p^{-1}$) vanishes identically 
if $\det H$ is constant on $X$.
\end{Proposition}

In particular, when $k=1$ and $\omega=c_1(L,h^L)$ we obtain the asymptotics 
for the expectation of the currents of integration along simultaneous zeros as $p\to\infty$:
\begin{equation} \label{eq:7.7}
\frac{1}{p^r}\,\E[s_p=0]=\frac{1}{p^r}\,\E[s_{p,1}=\ldots=s_{p,r}=0]
=\omega^r-\frac{i}{2\pi p}\,\partial\bar{\partial}\log \det H 
\wedge \omega^{r-1}+O\left(\frac{1}{p^2}\right),
\end{equation}
where $s_p=(s_{p,1},\ldots, s_{p,r})\in V_p$.

If we assume that the function $H$ is the constant identity matrix $H \equiv\mathrm{Id}_r$ 
then the random sections $s_{p,1}, \ldots, s_{p,r}$ are mutually independent identically 
distributed Gaussian variables. Consequently, Proposition \ref{P:r-tuple} $(i)$ 
and \eqref{eq:7.7} precisely replicate the results obtained by Shiffman-Zelditch 
\cite[Propositions 4.4 and 4.5]{ShZ99} and Dinh-Sibony \cite[Section 7]{DS06}.
It is noteworthy that Dinh-Sibony \cite[Section 7]{DS06} also considered a broader range of 
probability measures on sections beyond Gaussian or Fubini-Study measures.

Letting $k=r$ in \eqref{eq:7.2-2} we obtain 
\begin{equation}
\label{eq:7.2-3}
\frac{1}{p}\,\E[D_r(s^1_p,\ldots, s^r_p)]=
r\omega-\frac{i}{2\pi p}\,\partial\bar{\partial}\log \det H+
O\left(\frac{1}{p^2}\right),\quad p\to\infty.
\end{equation}
Let $\det(s^1_p,\ldots, s^r_p)$ be the determinant of 
$(s^1_p,\ldots, s^r_p)$, which is a holomorphic section of 
$L^{pr}\otimes \Lambda^r E_0$ (see Section \ref{S:tuple}). 
Note that $\Lambda^r E_0=X\times\C$ is the trivial line bundle on $X$. 
If we equip it with the Hermitian metric $h_H$ such that 
$|1|^2_{h_H}(x):=\det H(x)$, then we can rewrite \eqref{eq:7.2-3} as
\begin{equation}
\label{eq:7.2-4}
\frac{1}{p}\,\E[D_r(s^1_p,\ldots, s^r_p)]=
c_1(L^{r}\otimes \Lambda^r E_0, h^{L^r}\otimes h^{1/p}_H)+O\left(\frac{1}{p^2}\right).
\end{equation}

If we take $H\equiv\mathrm{Id}_r$ then we are in the setting of Section \ref{S:tuple}, that is, $s^1_p,\ldots, s^r_p$ are independent and identically distributed Gaussian variables. By \eqref{eq:5.17new} and under the assumption that $\omega=c_1(L,h^L)$, we get a special case of \eqref{eq:7.2-4},
\begin{equation}
\label{eq:7.2-5}
\begin{split}
    \frac{1}{p}\,\E[D_r(s^1_p,\ldots, s^r_p)]_{H=\mathrm{Id}_r}=\frac{r}{p}\,\Phi_{L^p}^\star(\omega_{\mathrm{FS}})=c_1(L^{r}, h^{L^r})+O\left(\frac{1}{p^2}\right),
\end{split}
\end{equation}
where the last identity is deduced from Tian's original theorem for positive line bundles.

\medskip

We conclude the paper by considering constant metrics on $E_0$. 
In this instance, we compute the covariance between the 
components of a random vector $s_p = (s_{p,1}, \ldots, s_{p,r}) \in V_p$.

\begin{Proposition}\label{prop:7.1}
Given $A=(a_{j\ell})_{1\leq j,\ell\leq r}\in M_+(\C^r)$ we 
equip $E_0$ with the Hermitian metric  $h^{E_0}_A(u,v):=h^{E_0}_0(A^{-1}u,v)$. 
Let $\mathcal{G}_p$ denote the Gaussian probability measure on 
$V_p$ induced by $h^L$ and $h^{E_0}_A$. Let 
$s_p=(s_{p,1},\ldots, s_{p,r})$ be a random element in $V_p\,$,
where each $s_{p,j}=\Pi^p_j(s_p)$ is a random variable valued in 
$H^0(X,L^p)$. Then, for all sufficiently large $p\in\N$, we have:

(i) Each $s_{p,j}$ is a Gaussian random variable valued in $H^0(X,L^p)$ with the 
same distribution as $\sqrt{a_{jj}}s_{p,0}\,$ (note that $a_{jj}>0$). Here $s_{p,0}$
denotes the random element in $H^0(X,L^p)$ with the Gaussian probability 
measure induced by the $L^2$-inner product associated with $h^L$.

(ii) The covariance matrix is given by
\begin{equation}
    \E[(s_{p,j},s_{p,\ell})_p]=n_p a_{j\ell }, \text{ for } 1\leq j,\ell\leq r,
\end{equation}
where $n_p:=\dim H^0(X,L^p)$.
\end{Proposition}

\begin{proof}
Let $(\cdot,\cdot)_{p,L^2(h^{E_0}_0)}$, $(\cdot,\cdot)_{p,L^2(h^{E_0}_A)}$ 
denote the $L^2$-inner products on $V_p$ induced by $h^{E_0}_0$ and 
$h^{E_0}_A$ respectively. Let $A_p\in\mathrm{End}(V_p)$ be defined by 
\[A_p s_p=\left(\sum_{j=1}^r a_{1j}s_{p,j}, \ldots, 
\sum_{j=1}^r a_{rj}s_{p,j}\right)\]
with respect to the splitting \eqref{eq:7.1}. Then for $s_p,s'_p\in V_p$, we have
\begin{equation}
    (s_p,s'_p)_{p,L^2(h^{E_0}_0)}=(A_p s_p,s'_p)_{p,L^2(h^{E_0}_A)}.
    \label{eq:7.3A}
\end{equation}
For $j\in\{1,\ldots,r\}$, let $E^p_j: H^0(X,L^p)\rightarrow V_p$ 
be the linear embedding that gives the $j$-th component in the 
splitting \eqref{eq:7.1}. Let $\{e_{p,k}\}_{k=1}^{n_p}$ 
be an orthonormal basis of $H^0(X,L^p)$ with respect to the 
$L^2$-inner product $(\cdot,\cdot)_p$ induced by $h^L$. Then we have
\[ s_{p,0}=\sum_{k=1}^{n_p}\eta_k e_{p,k},\]
where $\{\eta_k\}_k$ is a vector of independent and identically 
distributed standard complex Gaussian variables (with complex variance one).

Let $s_p$ be the random element in $V_p$ equipped with the 
Gaussian probability measure $\mathcal{G}_p$ (induced by $h^{E_0}_A$). 
We proceed to prove (i) and (ii).
Since $s_{p,j}=\Pi^p_j(s_p)$ is a linear transformation of a Gaussian vector, 
$s_{p,j}$ is a Gaussian random variable valued in $H^0(X,L^p)$. 
That is, for any $v\in H^0(X,L^p)$, $(s_{p,j},v)_p$ is a complex Gaussian variable.
To determine the distribution of $s_{p,j}$, it suffices to study 
the Gaussian variables $(s_{p,j},e_{p,k})_p$ for $k=1,\ldots, n_p$.
By the definitions of $h^{E_0}_A$ and $A_p$, we have the following identity
\begin{equation}
(A_pE^p_j e_{p,k'},A_p E^p_{j'} e_{p,k})_{p, L^2(h^{E_0}_A)}=
\delta_{k'k}a_{j'j}\,.
\label{eq:7.4}
\end{equation}
Using \eqref{eq:7.3A} and the fact that $A$ is Hermitian we obtain
\begin{equation*}
    \begin{split}
        (s_{p,j},e_{p,k})_p&=(\Pi^p_j(s_p),e_{p,k})_p=(s_{p},E^p_j e_{p,k})_{p, L^2(h^{E_0}_0)}\\
        &=(A_ps_{p},E^p_j e_{p,k})_{p, L^2(h^{E_0}_A)}=(s_{p},A_pE^p_j e_{p,k})_{p, L^2(h^{E_0}_A)}.
    \end{split}
\end{equation*}
Therefore,
\begin{equation}\label{eq:7.6}
s_{p,j}=\sum_{k=1}^{n_p}
(s_{p},A_pE^p_j e_{p,k})_{p, L^2(h^{E_0}_A)} e_{p,k}\,.
\end{equation}
By \eqref{eq:7.4}, for a fixed $j \in \{1, \ldots, r\}$, the set
$\left\{ \frac{1}{\sqrt{a_{jj}}}A_pE^p_j e_{p,k}\right\}_{k=1}^{n_p}$
forms an orthonormal set in $V_p$ with respect to the inner product
$(\cdot,\cdot)_{p,L^2(h^{E_0}_A)}$. 
By the fact that $s_p$ is Gaussian with respect to this inner product, 
we obtain that 
 \[ \left\{ \frac{1}{\sqrt{a_{jj}}}(s_{p},A_pE^p_j 
 e_{p,k})_{p, L^2(h^{E_0}_A)}\right\}_{k=1}^{n_p}\]
is a vector of independent and identically 
distributed standard complex Gaussian variables, 
see \cite[Lemma 2.4.2 and Theorem 2.4.3]{Stroock}. 
Thus (i) follows from the definition of $s_{p,0}$ and 
\eqref{eq:7.6}. Moreover, we have
\[ (s_{p,j},s_{p,\ell})_p=\sum_{k=1}^{n_p} 
(s_{p},A_pE^p_j e_{p,k})_{p, L^2(h^{E_0}_A)} 
\overline{(s_{p},A_pE^p_\ell e_{p,k})}_{p, L^2(h^{E_0}_A)}.\]
Consequently, \eqref{eq:7.4} and the fact that $s_p$ is Gaussian directly imply (ii).
\end{proof}

\end{document}